\documentclass[11pt]{amsart}
\usepackage{amsmath,amssymb,amsthm,enumerate}
\usepackage[latin1]{inputenc}
\usepackage{version,tabularx,multicol}
\usepackage{graphicx,float,url}
\usepackage{marginnote}
\usepackage{mathrsfs}
\newcommand{\reff}[1]{(\ref{#1})}

\theoremstyle{plain}
\newtheorem{theo*}{Theorem}
\newtheorem{prop*}[theo*]{Proposition}
\newtheorem{cor*}[theo*]{Corollary}
\newtheorem{theo}{Theorem}[section]

\newtheorem{cor}[theo]{Corollary}
\newtheorem{prop}[theo]{Proposition}
\newtheorem{proposition}[theo]{Proposition}
\newtheorem{lem}[theo]{Lemma}
\newtheorem{defi}[theo]{Definition}
\theoremstyle{definition}

\theoremstyle{remark}
\newtheorem{rem}[theo]{Remark}

\newcommand{\cc}{{\mathcal C}}

\newcommand{\ci}{{\mathcal I}}
\newcommand{\cj}{{\mathcal J}}
\newcommand{\ck}{{\mathcal K}}

\newcommand{\cn}{{\mathcal N}}

\newcommand{\ct}{{\mathcal T}}

\newcommand{\E}{{\mathbb E}}

\newcommand{\N}{{\mathbb N}}
\renewcommand{\P}{{\mathbb P}}

\newcommand{\R}{{\mathbb R}}

\newcommand{\rP}{{\rm P}}
\newcommand{\rE}{{\rm E}}
\newcommand{\bN}{{\mathbf N}}
\newcommand{\bP}{{\mathbf P}}
\newcommand{\bE}{{\mathbf E}}

\newcommand{\ind}{{\bf 1}}

\newcommand{\Card}{{\rm Card}\;}

\newcommand{\inv}[1]{\mathop{\frac{1}{ #1}}\nolimits}
\newcommand{\expp}[1]{\mathop {\mathrm{e}^{ #1}}}

\newcommand{\llambda}{\Lambda}

\newcommand{\lb}{[\![}
\newcommand{\rb}{]\!]}

\begin{document}

\title{Smaller population size at the MRCA time for stationary
  branching processes}

\date{\today}

\author{Yu-Ting Chen}
\address{Yu-Ting Chen
Department of Mathematics, The University of British Columbia, 1984 Mathematics Road, Vancouver, B.C., Canada V6T 1Z2}
\email{ytchen@math.ubc.ca}

\author{Jean-Fran\c{c}ois Delmas}
\address{
Jean-Fran\c cois Delmas,
Universit\'e Paris-Est, CERMICS,  6-8
av. Blaise Pascal,
  Champs-sur-Marne, 77455 Marne La Vall\'e, France.
\url{http://cermics.enpc.fr/~delmas/home.html}
}
\email{delmas@cermics.enpc.fr}

\thanks{This work is partially supported by National Center for
  Theoretical Science Mathematics Division and NSC 97-2628-M-009-014,
  Taiwan and by the French ``Agence Nationale de
  la Recherche'', ANR-08-BLAN-0190.}

\begin{abstract}
  We present an elementary model of random size varying population given
  by a stationary continuous state branching process.  For this model we
  compute the joint distribution of:  the time to the most recent common
  ancestor,  the size  of the  current population  and the  size  of the
  population  just before  the most  recent common  ancestor  (MRCA). In
  particular  we show a  natural mild bottleneck  effect as  the size  of the
  population  just before the  MRCA is  stochastically smaller  than the
  size of  the current  population.  We also  compute the number  of old
  families which  corresponds to the  number of individuals  involved in
  the last coalescent event of  the genealogical tree.  By studying more
  precisely  the genealogical structure  of the  population, we  get
  asymptotics for  the number  of ancestors just  before the
  current  time.  We  give  explicit  computations in  the  case of  the
  quadratic  branching  mechanism.   In  this  case,  the  size  of  the
  population  at the MRCA  is, in  mean, less  by 1/3  than size  of the
  current population size.  We also provide in this case the fluctuations for
  the renormalized number of ancestors.
\end{abstract}



\maketitle

\section{Introduction}

A large  literature is  devoted to constant  size population  models. It
goes back  to Wright \cite{w:emp} (1930) and  Fisher \cite{f:gtn} (1931)
in  discrete   time,  and   Moran  \cite{m:rpg}  (1958)   in  continuous
time. Models  for constant infinite  population in continuous  time with
spatial motion  were introduced  by Fleming and  Viot \cite{fl:smvmppgt}
(1979).   On the  other  hand, the  study  of the  genealogical tree  of
constant size population was initiated by Kingman \cite{k:c} (1982), and
described in  a more general  setting by Pitman \cite{p:cmc}  (1999) and
Sagitov  \cite{s:gcamal}  (1999).    The  complete  description  of  the
genealogy of  the Fleming-Viot process  can be partially done  using the
historical super-process  by Dawson and Perkins  \cite{dp:hp} (1991) and
precisely by using the look-down process developed by Donnelly and Kurtz
\cite{dk:crfvmvd,dk:prmvpm} (1999) or  the stochastic flows from Bertoin
and Le Gall \cite{blg:sfacpI,blg:sfacpII,blg:sfacpIII} (2003).

It  is  however  natural  to  consider random  size  varying  population
models. Branching  population models, for which sizes  of the population
are  random, goes  back to  Galton  and Watson  \cite{gw:pef} (1873)  in
discrete time  and with finite mass  individual.  Jirina \cite{j:sbpcss}
(1958)  considered  continuous   state  branching  process  (CB)  models
corresponding to individuals with  infinitesimal mass.  The genealogy of
those  processes  can  be  partially described  through  the  historical
super-process. However  the continuum L\'evy tree introduced  by Le Gall
and Le Jan \cite{lglj:bplpep} (1998) and developed later by Duquesne and
Le Gall \cite{dlg:rtlpsbp} (2002)  allows to give a complete description
of  the genealogy  in  the  critical and  sub-critical  cases.  See  the
approach  of  Abraham and  Delmas  \cite{ad:cmp}  (2008) or  Berestycki,
Kyprianou and  Murillo \cite{bkm:pbss} (2009)  for a description  of the
genealogy in the super-critical cases.

The  two families  of  models: constant  size  population and  branching
populations  are, in  certain cases,  linked.  The  case of  a quadratic
branching corresponds  to the fact  that only two genealogical  lines of
the population genealogical tree can merge together.  In this particular
case,  it is  possible  to  establish links  between  the constant  size
population  model  and  CB  models.   Thus, conditionally  on  having  a
constant  population  size,   the  Dawson-Watanabe  super-process  is  a
Fleming-Viot process,  see Etheridge-March \cite{em:ns}  (1991).  On the
other hand, using a time  change (with speed proportional to the inverse
of  the population  size),  it  is possible  to  recover a  Fleming-Viot
process    from   a    Dawson-Watanabe   super-process,    see   Perkins
\cite{p:cdwpfvp}  (1992).  Birkner,  Blath,  Capaldo, Etheridge,  Möhle,
Schweinsberg  and  Wakolbinger  \cite{bbcemsw:asbbc} (2005)  have  given
similar results for stable branching  mechanism. In the same spirit, Kaj
and Krone  \cite{kk:cppsv} (2003) studied the  genealogical structure of
models of random size varying  population models and recover the Kingman
coalescent with a random time change.

Recently, some  authors studied the coalescent  process (or genealogical
tree) of  random size  varying population, in  this direction  see Möhle
\cite{m:cpte} (2002), Lambert \cite{l:ctbp} (2003) for branching process
and Jagers and Sagitov \cite{js:ccpsvs} (2004) for
stationary random size  varying population. 

Our primary  interest is to present  an elementary model  of random size
varying population and exhibit some interesting property which could not
be observed  in constant size model.   The most striking  example is the
natural mild bottleneck effect: in  a stationary regime, the size of the
population  just  before  the  most  recent common  ancestor  (MRCA)  is
stochastically smaller than the current population size. Our second goal
is to give  some properties of the coalescent tree such  as: time to the
most recent  common ancestor (TMRCA), number of  individuals involved in
the last coalescent  event, asymptotic behavior of the  number of recent
ancestors.

One of  the major  drawback of the  branching population models  is that
either the population  becomes extinct or decreases to  0, which happens
with probability 1 in the (sub)critical cases, or blows up exponentially
fast with positive probability in the super-critical case. In particular
there  is no  stationary regime,  and the  study of  the genealogy  of a
current population  depends on the  arbitrary original size and  time of
the  initial population.   To  circumvent this  problem,  we consider  a
sub-critical  CB,  $Y=(Y_t,t\geq 0)$,  with  branching mechanism  $\psi$
given by \reff{eq:def-psi}.  We get the Q-process by conditioning $Y$ to
non-extinction   (which  is   an   event  of   zero  probability),   see
\cite{l:abctsbp} and \cite{l:qdcsbpcne}.  The Q-process can also be seen
as a CB with immigration, see \cite{rr:pdwcfl}.  We take the opportunity
to present a probabilistic construction of independent interest for the
Q-process  in  Corollary \ref{cor:william}  which  relies  on a  Williams'
decomposition of CB described in \cite{ad:wdlcrtseppnm}. A first study
of the genealogical tree of the Q-process can be found in
\cite{l:ctbp}. 

We consider the Q-process  under its stationary distribution and defined
on the real line: $Z=(Z_t, t\in \R)$. Its Laplace transform, see \reff{lem:Laplace-Z}, is given by 
\[
\E\left[\expp{-\lambda Z_t}\right]=\exp\left(-\int_0^\infty
ds\;  \tilde \psi'(u(\lambda,s))  \;
\right), \quad \lambda\geq 0,\quad t\in \R,
\]
where $\tilde{\psi}(\lambda)=\psi(\lambda) - \lambda \psi'(0)$. 
In order for $Z_t$ to be finite, we shall assume condition
\reff{eq:A2}: 
\[
\int_0^1 \left(\frac{1}{
      v\psi'(0)}-\frac{1}{\psi(v)}\right)dv<+\infty.
\]
In order for the  TMRCA to be finite, we  assume condition \reff{eq:A1}:
\[
\int_1^\infty  \frac{dv}{\psi(v)}<+\infty.
\]
Notice  a  very  similar  condition exists  to  characterize  coalescent
processes which  descent from  infinity, see \cite{bbl:lcscdi}.  

As  in the  look-down representation  for constant  size  population, we
shall  represent  the process  $Z$  using  the  picture of  an  immortal
individual which gives birth to independent sub-populations or families.
For fixed time $t_0=0$ (which we can indeed choose to be equal to $0$ by
stationarity), we  consider $A$  the TMRCA of  the population  living at
time $0$,  $Z^A=Z_{(-A)-}$ the  size of the  population just  before the
MRCA,  $Z^I$ the  size of  the  population at  time $0$  which has  been
generated by the immortal individual over the time interval $(-A,0)$ and
$Z^O=Z_0-Z^I$  the size of  the population  at time  $0$ which  has been
generated  by  the  immortal   individual  at  time  $-A$.   In  Theorem
\ref{theo:(Z,ZA,ZI,ZO)-dsitrib},  we  give  the  joint  distribution  of
$(A,Z^A,Z^I,Z^O)$.     One   interesting    phenomenon    is   Corollary
\ref{cor:Indep-AZAZ}.
\begin{cor*}
Conditionally on
$A$;   $Z^A$,   $Z^I$  and   $Z^O$   are   independent. 
\end{cor*}
In   particular,    conditionally   on   $A$,   $Z^A$    and   $Z$   are
independent. Conditionally on $A$, $Z^A$ depends on the past before $-A$
of the process $Z$ and has to  die at time $0$, $Z^O$ corresponds to the
size of  the population  at time  $0$ generated at  time $-A$  and $Z^I$
corresponds to the  size of the population at time  $0$ generated by the
immortal  individual  over the  time  interval  $(-A,0)$.  Then, as  the
immortal  individual   gives  birth  to   independent  populations,  the
Corollary is then intuitively clear.

One of the most striking result, the natural mild bottleneck effect, is
stated in Proposition~\ref{theo:bottleneck}:
\begin{prop*}
   $Z^A$ is stochastically smaller than $Z_0$. 
\end{prop*}
Thus just before  the MRCA, the population size is unusually
small.  Notice this  result is not true in general  if one considers the
size of  the population at   the MRCA instead of just before, see
Remark \ref{rem:ZA+}.   We get nice quantitative results for  the quadratic
branching mechanism case, see Corollary \ref{cor:Z-ZA-quad}.
\begin{cor*}
  Assume $\psi$ is quadratic (and given by
 \reff{eq:psi-quadratic}). We have: a.s.
\[
   \P(Z^A<Z_0|A)=\frac{11}{16} \quad\text{and}\quad \E[Z^A|A]=\frac{2}{3} \E[Z_0|A]
\]
an in particular:
\[
\P(Z^A<Z_0)=\frac{11}{16}    \quad\text{and}\quad    \E[Z^A]=\frac{2}{3}
\E[Z_0].
\]
\end{cor*}
Notice that even is $Z^A$ is stochastically smaller than $Z_0$ it is not
a.s. smaller. 

We also give in Theorem \ref{theo:Z,An0}  the joint distribution of $Z_0$ and
the  TMRCA of  the immortal  individual  and $n$  individuals picked  at
random in the population at time $0$. See also related results in 
\cite{l:ctbp}. 

We investigate  in Proposition  \ref{prop:NA} the joint  distribution of
$A,Z_0$ and  $N^A$, where $N^A+1$  represents the number  of individuals
involved in the last coalescent  event of the genealogical tree. Under a
first moment condition  on $Z$, we get that if the  TMRCA is large, then
the last coalescent event is  likely to involve only two individuals. In
the stable case,  this first moment condition is  not satisfied, and the
last  coalescent  event  does  not  depend  on  the  TMRCA,  see  Remark
\ref{rem:stable-NA}. This suggests a  result similar to the one obtained
in  \cite{bbcemsw:asbbc}:  in  the  stable  case, the  topology  of  the
genealogical tree  (which does not take  into account the  length of the
branches) may not depend on its depth given by the TMRCA. 

After giving  a more precise description  of the genealogy  of $Z$ using
continuum  L\'evy  trees,  we  compute in  Theorem  \ref{prop:CVNt}  the
asymptotic  behavior of  the number  of ancestors  at time  $-s$, $M_s$,
of the
population at time $0$.
\begin{theo*}
The following convergence holds in probability:
\[
\lim_{s\downarrow 0} \frac{M_s}{c(s)}=Z_0, 
\]
where $c(s)$ is related to the extinction probability of $Y$ and defined
by $\displaystyle \int_{c(t)}^\infty \frac{dv}{\psi(v)}=t$.
\end{theo*}
This result is very similar to  the one obtained on coalescent process in
\cite{bbl:lcscdi} (notice the convergence is a.s. in \cite{bbl:lcscdi}).
We can precise the fluctuations in the quadratic case, see
Theorem~\ref{theo:cvNpi=0}. 
\begin{theo*}
Assume $\psi$ is quadratic (and given by
 \reff{eq:psi-quadratic}). We have
\[
\sqrt{c(s)\E[Z_0]}\left(\frac{M_s}{c(s)} -Z_0\right)
\xrightarrow[s\downarrow   0+]{(\rm   d)} (Z_0-Z'_0),
\]
where  $Z'_0$ is
distributed as $Z_0$ and independent of $Z_0$.
\end{theo*}

\bigskip

The paper is  organized as follows. We first recall  well known facts on
CB in  Section \ref{sec:CB}. We  introduce in Section  \ref{sec:SCB} the
corresponding stationary  CB, which is  related to the Q-process  of the
CB, and  give its  first properties. We  give the joint  distribution of
$(A,Z^A,Z^I,  Z^O)$ in  Section  \ref{sec:TMRCA} and  prove the  natural
bottleneck  effect,  that  is   $Z^A$  is  stochasitcally  smaller  than
$Z_0$. We compute  the number of old families  (or number of individuals
involved in the last coalescent  event) in Section \ref{sec:nof} and the
asymptotics of  the number of ancestors  in Section \ref{sec:ancestors}.
A  first  consequent part  of  this latter  Section  is  devoted to  the
introduction  of the genealogy  of CB  processes using  continuum random
L\'evy trees. We  give more detailed results in  the quadratic branching
setting of Section \ref{sect:ex}.

\section{Continuous-state branching process (CB)}
\label{sec:CB}
We  recall some  well-known fact  on continuous-state  branching process
(CB), see for example \cite{l:mvbmp} and references therein. We consider  a \textbf{sub-critical} branching mechanism $\psi$:
for $\lambda\geq 0$,
\begin{equation}
   \label{eq:def-psi}
\psi(\lambda)=\alpha \lambda + \beta \lambda^2 +\int_{(0,+\infty )}
\pi(d\ell) \left[\expp{-\lambda \ell} -1+ \lambda \ell
 \right], 
\end{equation}
where $\alpha=\psi'(0)>0$, $\beta\geq 0$ and $\pi$ is a Radon measure on
$(0,{+\infty} )$  such that $\int_{(0,{+\infty} )}  (\ell \wedge \ell^2)
\; \pi(d\ell)<{+\infty}  $.  We  consider the non  trivial case  that is
either  $\beta>0$  or $\pi((0,1  ))=+\infty  $.  Notice  that $\psi$  is
convex, of class $\cc^1$ on $[0,+\infty )$ and of class $\cc^\infty $ on
$(0,+\infty )$ and $\psi''(0+)\in (0, +\infty ]$.

Let  $\rP_x$ be  the  law of  a  CB $Y=(Y_t,t\ge  0)$
started  at mass $x\geq  0$  and  with branching  mechanism  $\psi$, and  let
$\rE_x$ be the corresponding  expectation. The process $Y$ is
a c\`ad-l\`ag  $\R_+$-valued Feller process and $0$ is a cemetery point.
The process $Y$ has no fixed discontinuities.   For
every $\lambda>0$, for every $t\ge 0$, we have
\begin{equation}
\label{eq:laplace_csbp}
\rE_x\left[\expp{-\lambda Y_t}\right]=\expp{-xu(\lambda,t)},
\end{equation}
where the function $u$ is the unique non-negative solution of
\begin{equation}
\label{eq:int_u}
u(\lambda,t)+\int_0^t\psi\bigl(u(\lambda,s)\bigr)ds=\lambda, \quad
\lambda\geq 0, \quad t\geq 0.
\end{equation}
Note that the function $u$ is equivalently characterized as the unique
non-negative solution of 
\begin{equation}
   \label{eq:int_u2}
\int_{u(\lambda,t)}^\lambda \frac{dr}{\psi(r)}=t \quad
\lambda\geq 0, \quad t\geq 0.
\end{equation}
or as the unique non-negative solution of: for $\lambda\geq 0$, 
\begin{equation}
\left\{
\begin{array}{ll}
\partial_t u + \psi(u)&=0\quad t> 0,\\
u(\lambda,0)&=\lambda.
\end{array}
\right.\label{eq:ODE-u}
\end{equation}
Markov property of $Y$ implies that for all $\lambda,s, t\geq 0$:
\begin{equation}
   \label{eq:uu}
u(u(\lambda,t),s)=u(\lambda,t+s).
\end{equation}

Let  $\N$ be  the canonical  measure (we  shall also  call  it excursion
measure) associated to $Y$.  It is a $\sigma$-finite measure which
intuitively  describe  the  distribution  of  $Y$  started  at  an
infinitesimal mass. We recall that if 
\[
\sum_{i\in I} \delta_{x_i,
  Y^{i}} (dx, 
dY)
\]
is a Poisson point measure with intensity $\ind_{[0,+\infty
  )}(x)\; dx \N[dY]$, then 
\begin{equation}
   \label{eq:Y-rep}
\sum_{i\in I} \ind_{\{x_i\leq x\}}
Y^{ i} 
\end{equation}
 is distributed as $Y $ under $\rP_x$. In particular, we have: for
 $\lambda\geq 0$ 
 \[
 \N\left[1-\expp{-\lambda Y_t}\right]=\lim_{x\downarrow 0}
   \inv{x}\rE_x\left[1-\expp{-\lambda Y_t}\right] =u(\lambda,t).
 \]
For convenience, we shall put $Y_t=0$ for $t<0$.

Let
$\zeta=\inf\{ t; Y_t=0\}$ be the extinction time of $Y$. 
We consider the function:
\begin{equation}
   \label{eq:def-c}
   c(t)=\N[\zeta>t]=\N[Y_t>0]=\lim_{\lambda\rightarrow \infty }\uparrow u(\lambda
   ,t). 
\end{equation}
We shall assume throughout this paper, but for Sections \ref{sec:ppm}
and \ref{sec:Z1}, 
 that the strong extinction property
holds:
\begin{equation}
   \label{eq:A1}\tag{A1}
\int_1^\infty  \frac{dv}{\psi(v)}<+\infty.
\end{equation}
It follows from \reff{eq:int_u2} and  \reff{eq:def-c} that $c$ is the
unique non-negative solution of:
\begin{equation}
   \label{eq:int-c}
\int_{c(t)}^\infty \frac{dv}{\psi(v)}=t, \quad t>0.
\end{equation}
Thanks to \reff{eq:A1},  we get that $c(t)$ is finite  for all $t>0$ and
$\N[\zeta=+\infty ]=0$.  We also get  that $c$ is  continuous decreasing
and thus one-to-one from $(0,+\infty )$ to $(0,+\infty )$. Letting
$\lambda$ goes to infinity in \reff{eq:uu} yields that for $s,t\geq 0$
\begin{equation}
   \label{eq:uc=c}
u(c(t),s)=c(t+s).
\end{equation}

\section{Stationary CB}
\label{sec:SCB}

In  contrast to  Wright-Fisher  population models,  CB  models do  not
exhibit stationary distributions.  However, by conditioning sub-critical
CB  to  non-extinction  (see \cite{rr:pdwcfl},  \cite{ep:mvmbpcn}  and
\cite{l:qdcsbpcne} for details), one  get the so-called Q-process, which
we  denotes  by  $Y''$.   This  process  is also  a  CB  process  with
immigration  introduced in  \cite{kw:bpirll} and  may have  a stationary
distribution.  This process, as  pointed out  in \cite{a:crt2}  see also
\cite{e:trcs},  has a  heuristic interpretation  by introducing  a fixed
ancestral lineage,  namely it is an  independent sum of  the process $Y$
and  the population  thrown off  by  an "immortal  individual" whose  laws
coincide with the law of a generic population $Y$.

We introduce the  process $Y''$ in Section \ref{sec:ppm}  as well as its
stationary version $Z$. Then  we check in Section \ref{sec:q-proc}, that
under \reff{eq:A1} the process  $Y''$ is indeed the Q-process associated
to  $Y$.  This gives  then  a natural  interpretation  of  $Z$. We  give
preliminary results on the process $Z$ in Sections \ref{sec:Z1} and
\ref{sec:Z2}.

\subsection{Poisson point measure of CB}
\label{sec:ppm}
We consider the following Poisson point measures.
\begin{itemize}
\item Let  $\cn_0(dr,dt)=\sum_{i\in I} \delta_{(r_i,  t_i)}(dr,dt)$ be a
  Poisson point measure on $(0,+\infty )\times \R$ with intensity
\[
r\; \pi(dr)dt.
\]
\item   Conditionally  on   $\cn_0$, let $(\cn_{1,i}, i\in I)$, where 
$\cn_{1,i}(dt,dY)=\sum_{j\in  J_{1,i}}
  \delta_{t_j,  Y^j} (dt, dY)$, 
 be independent
  Poisson point measures with respective intensity 
\[
r_i   \delta_{t_i}(dt)\N[dY].
\]
Notice  that  for  all  $j\in  J_{1,i}$,  we  have  $t_j=t_i$.   We  set
$J_1=\bigcup  _{i\in  I}   J_{1,i}$  and  $\cn_1(dt,dY)=\sum_{j\in  J_1}
\delta_{t_j, Y^j} (dt, dY)$.
\item Let $\cn_2(dt,dY)=\sum_{j\in J_2} \delta_{t_j, Y^j} (dt, dY)$ be a
  Poisson  point   measure  independent  of   $(\cn_0,\cn_1)$  and  with
  intensity
\[
2 \beta \; dt\;  \N[dY].
\]
\end{itemize}
We set $\cj=J_1\bigcup J_2$. We
shall call  $Y^j$, with $j\in \cj$ a  {\bf
  family}  and $t_j$  its birth  time.  

We will  consider the two following processes $Y''=(Y''_t, t\geq 0)$ and
its stationary version $Z=(Z_t, t\in \R)$: 
\begin{align}
\label{eq:defY2}
Y''_t&=\sum_{j\in \cj, t_j>0} Y_{t-t_j}^j,\\
   \label{eq:def-Z}
Z_t&=\sum_{j\in \cj} Y_{t-t_j}^j.
\end{align}
We will denote by $\P$ the probability under which $Y''$ and $Z$ are
defined and $\E$ the corresponding expectation. 

At this stage,  let us emphasize there is  another natural decomposition
of $Y''$ and $Z$. For  $i\in I$, set $Y^i=\sum_{j\in J_{1,i}} Y^j$
and $\ci=I\bigcup J_2$. The random measure 
\begin{equation}
   \label{eq:defn3}
 \cn_3(dt,dY)=\sum_{i\in  \ci} \delta_{t_i,  Y^i} (dt, dY)
 \end{equation} 
 is a Poisson point measure with intensity $dt \mu(dY)$ and 
\begin{equation}
   \label{eq:def-mu}
\mu(dY)= 2 \beta \N[dY]+ \int_{(0,+\infty )} \ell
\pi(d\ell) \; \rP_\ell(dY).
\end{equation}
And we have:
\begin{align}
   \label{eq:def-Y22}
Y''_t&=\sum_{i\in \ci; t_i>0} Y^i_{t-t_i},\\
   \label{eq:def-Z2}
Z_t&=\sum_{i\in \ci} Y^i_{t-t_i}.
\end{align}
We shall  call $Y^i$, with $i\in  \ci$ a {\bf  clan} and $t_i$
its  birth  time.   For  $j\in  J_2$, $Y^j$  is  a  clan  and  a
family.  Notice that  a.s.  two  clans  have different
birth time, but families in the same clan  have the same birth time.

The presentation with clans is  simpler than the representation with
families and most of the results can be obtained using the former
representation. We will use the family representation in Sections
\ref{sec:nof} and \ref{sec:ancestors}.

We define $\tilde \psi$ by:
\begin{equation}
   \label{eq:def-psi'}
\tilde{\psi}(\lambda)=\psi(\lambda) -\lambda \psi'(0)=\psi(\lambda)
-\alpha \lambda. 
\end{equation}

We first give a Lemma on the family representation. 
\begin{lem}
   \label{lem:EFfam}
Let $F$ be a non-negative measurable function. We have
\begin{equation}
   \label{eq:Lapl-fam}
\E\left[\expp{-\sum_{j\in \cj} F(t_j, Y^j)}\right]
=\exp\left(-\int_\R dt \; \tilde \psi'(\N[1-\expp{ F(t, Y)}])\right).
\end{equation}
\end{lem}
\begin{proof}
  Using Poisson point measure properties, we get:
\begin{align*}
 \E\left[\expp{-\sum_{j\in \cj} F(t_j, Y^j)}\right]
&  =\E\left[\expp{-\sum_{j\in J_1} F(t_j,
      Y^j)}\right]\E\left[\expp{-\sum_{j\in J_2} F(t_j, Y^j)}\right] \\
&  =\E\left[\expp{-\sum_{i\in I} r_i \N[1-F(t_i,Y)]} \right]\expp{-2\beta
    \int     dt\;\N[1-\expp{ F(t, Y)}]} \\ 
&  =\expp{-\int dt  \int_{(0,+\infty )}  \ell \pi(d\ell)\; \left(1-
      \exp (-\ell \N[1-F(t,Y)])\right) } \expp{-2\beta
    \int     dt\;\N[1-\expp{ F(t, Y)}]} \\ 
&  =\expp{-\int dt \; \tilde \psi'(\N[1-\expp{ F(t, Y)}])}.
\end{align*}
\end{proof}

\begin{prop}
   \label{prop:CBI}
The process  $Y''$  is a  CB with  branching mechanism  $\psi$ and
immigration function $\tilde{\psi}'$:
\[
\tilde \psi'(\lambda)=2\beta \lambda + \int_{(0,+\infty )} \ell \pi(d\ell)\;
(1- \expp{-\lambda \ell})
\]
started at $Y''_0=0$. 
\end{prop}
\begin{proof}
   This is a direct consequence of Lemma \ref{lem:EFfam} and results from
   \cite{kw:bpirll}.  
\end{proof}

In particular  $Y''$ is a strong Markov process started at 0 and its
transition kernel is characterized by: for $\lambda\geq 0$,
$t\geq 0$, $r\geq 0$
\[
\E[\expp{-\lambda Y''_t}|Y''_0=r]=\exp {\left(- ru(\lambda,t) -  \int_0^t
  \tilde \psi'(u(\lambda,s))\; ds\right)}.
\]
The next result is then straightforward. 
\begin{cor}\label{cor:cbi}
For each $t\in \mathbb  R$, $\{Z_s;s\geq t\}$ has the same law as a
CB with branching mechanism $\psi$ and
immigration function $\tilde{\psi}'$ started at the invariant
distribution $\mathbb  P(Z_t\in \cdot)$. 
\end{cor}

\subsection{Q-process}
\label{sec:q-proc}
We check the process $Y''$  is indeed the
Q-process for CB using Williams'
decomposition.

Let  $m>0$ and  $\nu_m(dt)=\sum_{i\in I}  r_i  \delta_{t_i}(dt)$, where
$\sum_{i\in I}  \delta_{(r_i, t_i)}(dr,dt)$ is a  Poisson point measure
with intensity
\[
\ind_{[0,m]}(t) \expp{-rc(m-t)}\; r
\; \pi(dr)dt.
\]
Conditionally on $\nu_m$, let  $\sum_{j\in \cj^m} \delta_{t_j, Y^j} (dt,
dY)$ be a Poisson point measure with intensity
\[
\left(\nu_m(dt)+ 2\beta \ind_{[0,m]}(t)
\; dt\right) \; \N[dY, \zeta<m-t].
\]
The next Proposition is a consequence of Theorem 3.3 in
\cite{ad:wdlcrtseppnm}.  
\begin{prop}
   \label{prop:william}
Assume \reff{eq:A1} holds. Under $\N$, conditionally on
$\{\zeta=m\}$,   $Y$  is distributed as $(Y'_t, t\geq 0)$ where
\[
Y'_t=\sum_{j\in \cj^m} Y_{t-t_j}^j.
\]
\end{prop}

It is then easy to deduce the following Corollary using representation
\reff{eq:def-Y22} of $Y''$. 

\begin{cor}
   \label{cor:william} Assume \reff{eq:A1} holds. 
   The   limit  distribution   of  $Y$   under  $\N$,   conditionally  on
   $\{\zeta=m\}$,  as  $m$  goes  to  infinity, is  the  distribution  of
   $Y''$ from Proposition \ref{prop:CBI}.
\end{cor}

 Corollary \ref{cor:william}  readily implies that the Q-process
  associated to  $Y$, that is the  limit distribution of  $Y$ under $\N$,
  conditionally on  $\{\zeta\geq m\}$,  as $m$ goes  to infinity,  is the
  distribution of $Y''$ from Proposition \ref{prop:CBI}.

\subsection{Stationary CB}
\label{sec:Z1}

We first give an interpretation of $Z$ in population terms. At time $t$,
$Z_t$ correspond  to the size of  a population generated  by an immortal
individual (with zero mass) which  gives birth at rate $2\beta$ to clans
(or families) which sizes evolve  independently as $Y$ under $\N$ and at
rate 1 with intensity $r\; \pi(dr)$ to clans with initial size $r$ which
evolve  independently  as  $Y$   under  $\rP_r$.   

By construction the process $Z$ is stationary. 
The next Lemma which gives the Laplace transform of $Z$ is a direct
consequence of the construction of $Z$. 
\begin{lem}
   \label{lem:Laplace-Z}
For all $t\in \R$ and $\lambda\geq 0$, the Laplace transform of
$Z_t$ is given by: 
\begin{equation}
   \label{eq:Laplace-Z}
\E\left[\expp{-\lambda Z_t}\right]=\exp\left(-\int_0^\infty
ds\;  \tilde \psi'(u(\lambda,s))  \;
\right).
\end{equation}
\end{lem}
\begin{proof}
   Using Lemma \ref{lem:EFfam}, we have:
\[
   \E\left[\expp{-\lambda Z_t}\right]
= \exp\left(-\int_\R ds \; \tilde \psi'(\N[1-\expp{-\lambda Y_{t-s}}])\right)
=\exp\left(-\int_0^\infty
ds\;  \tilde \psi'(u(\lambda,s))  \;
\right).
\]
\end{proof}

We shall consider the following assumption
\begin{equation}
   \label{eq:A2}\tag{A2}
\int_1^{+\infty } \ell \log(\ell)\; \pi(d\ell)<+\infty .
\end{equation}
The  next Lemma  is well  known  (notice condition  \reff{eq:A1} is  not
assumed).
\begin{lem}
   \label{lem:equiv}
In the sub-critical case, 
the following conditions are equivalent:
\begin{enumerate}
   \item[(i)] \reff{eq:A2} holds.
\item[(ii)] $\int_0^1 \left(\frac{1}{\alpha
      v}-\frac{1}{\psi(v)}\right)dv<+\infty$. 
\item[(iii)] $\displaystyle \E_r[Y_t \log(Y_t)]<+\infty $ for some
  $t>0$ and $r>0$.
\item[(iv)] $\displaystyle \E_r[Y_t \log(Y_t)]<+\infty $ for all $t>0$
  and $r>0$. 
\end{enumerate}
\end{lem}
\begin{proof}
   For (i)  $\Leftrightarrow$ (ii)  see \cite{g:abctcsbp} proof  of Theorem
4a, and for (ii) $\Leftrightarrow$ (iii)  (or (iv)) use Lemma 1 p.25 of
\cite{an:bp}.
\end{proof}

The next Proposition gives a condition for finiteness of $Z$, see also
\cite{p:ltcsbpi} in a more general framework. 

\begin{proposition} We have  $\displaystyle
  \P(Z_0<+\infty )=1$  if and
  only if \reff{eq:A2} holds. 
\end{proposition}

\begin{proof}
  Thanks  to \reff{eq:Laplace-Z},  we get  $\displaystyle \P(Z_0<+\infty
  )=1$ if  and only  if $\lim_{\lambda\rightarrow 0}  \int_0^\infty ds\;
  \tilde  \psi'(u(\lambda,s)) =0$.   As  $\lambda \mapsto  u(\lambda,s)$
  decreases to $0$  as $\lambda$ goes down to $0$ for  all $s\geq 0$, we
  deduce  by dominated  convergence  that $\displaystyle  \P(Z_0<+\infty
  )=1$     if    and    only     if    $\int_0^\infty     ds\;    \tilde
  \psi'(u(\lambda,s))<+\infty $ for at least one  $\lambda>0$.

  Notice          that          $\partial_tu+\psi(u)=0$          implies
  $\psi'(u)=-\partial^2_tu/\partial_t   u$,    and   hence   for   every
  $0\leq t<T<+\infty$ we have
\begin{equation}
\label{eq:aux6} 
\int_{t}^T \tilde \psi'(u(\lambda,s))  \;
ds=\log\left(\frac{\psi(u(\lambda,t))\expp{\alpha
      t}}{\psi(u(\lambda,T))\expp{\alpha T}}\right) .
\end{equation}
We deduce that $T\mapsto \psi(u(\lambda,T))\expp{\alpha T}$ 
is decreasing. 
We also get that  $\int_0^\infty
  ds\; \tilde \psi'(u(\lambda,s))<+\infty $ if and only if
  $\lim_{T\rightarrow +\infty } \psi(u(\lambda,T))\expp{\alpha T}>0$ 
or equivalently  $\lim_{T\rightarrow +\infty } u(\lambda,T)\expp{\alpha
  T}>0$ as $\lim_{T\rightarrow+\infty } u(\lambda, T)=0$ thanks to
\reff{eq:int_u2}. 

We deduce from \reff{eq:int_u2} that
\begin{equation}
   \label{eq:uH}
u(\lambda,T)\expp{\alpha T} =\lambda \exp
\left(\alpha\int_{u(\lambda,T)}^\lambda dr\; \Big(\inv{\psi(r)}-\inv{\alpha
      r}\Big) \right).
\end{equation}
Thus we deduce  from  Lemma \ref{lem:equiv} that  $\displaystyle
  \P(Z_0<+\infty )=1$  if and
  only if \reff{eq:A2} holds. 
\end{proof}

\begin{cor}
   \label{cor:mom1Z}
Assume \reff{eq:A2} holds. We have for $\lambda>0$, $t\in \R$:
\begin{equation}
   \label{eq:mom1Z}
\E\left[Z_t\expp{-\lambda Z_t} \right]=\frac{\tilde
  \psi'(\lambda)}{\psi(\lambda)} \E\left[\expp{-\lambda Z_t} \right].
\end{equation}
In particular, we have: 
\begin{equation}
   \label{eq:mom1Z0}
\E[Z_t]=\frac{\psi''(0+)}{\psi'(0)}\in (0,+\infty ].  
\end{equation}
\end{cor}
\begin{proof}
We deduce from \reff{eq:Laplace-Z}   that:
\[
\E\left[Z_t\expp{-\lambda Z_t} \right]=\E\left[\expp{-\lambda
    Z_t}\right]\; 
\partial_\lambda \int_0^\infty
 \tilde \psi'(u(\lambda,s))  \;
  ds .
\]
We deduce from \reff{eq:int_u2} that $\lambda\mapsto u(\lambda,s)$ is
increasing and of
class $\cc^\infty $ on $(0,+\infty )$ and that
\begin{equation}
   \label{eq:dlu}
\partial_\lambda
u(\lambda,s)=\frac{\psi(u(\lambda,s))}{\psi(\lambda)}=\frac{-\partial_s
  u(\lambda,s)}{\psi(\lambda)}. 
\end{equation}
Thus, we get:
\begin{align*}
   \partial_\lambda \int_0^\infty
 \tilde \psi'(u(\lambda,s))  \;
  ds 
&=\int_0^\infty
  \psi''(u(\lambda,s))\partial_\lambda u(\lambda,s)  \;ds\\
&=- \inv{\psi(\lambda)}\int_0^\infty
  \psi''(u(\lambda,s))\partial_s
  u(\lambda,s) \;ds\\
&=\frac{\tilde \psi'(\lambda)}{\psi(\lambda)}.
\end{align*}
The last part of the Corollary is immediate.
\end{proof}

\begin{rem}
Assumption \reff{eq:A1} is  not needed  to define  the process
  $Y''$ or  the stationary process $Z$.  However the study  of MRCA for
  $Z$ is not relevant if \reff{eq:A1} does not hold.

  Notice, we will introduce a complete genealogical structure for $Z$ in
  Section \ref{sec:ancestors}  by using a genealogical  structure of the
  families $(Y^j, j\in \cj)$.
\end{rem}

\textbf{From now on, we shall assume that \reff{eq:A1} and \reff{eq:A2}
  are in force.} 

\subsection{Further property for stationary CB}
\label{sec:Z2}

By  construction,  we  deduce  that  for  all  $t\in  \R$,  the  process
$(Z_{s+t},  s\geq 0)$  is a  CB  with branching  mechanism $\psi$  and
immigration   function  $\tilde{\psi}'$   started   as  the   stationary
distribution  whose Laplace transform  is given  by \reff{eq:Laplace-Z}.
Then  Proposition 1.1  in \cite{kw:bpirll}  implies that  $Z$ is  a Hunt
process and in  particular it is c\`ad-l\`ag and  strongly Markov taking
values  in $[0,+\infty]$.  By  stationarity  and since  $+\infty  $ is  a
cemetery point for $Z$, we deduce  that a.s. for all $t\in \R$, $Z_t$ is
finite.

Next, we recall some asymptotic properties of the functions $u$ and $c$
given in Lemma 3.1 of \cite{l:ctbp}. 
\begin{lem}
\label{lem:AL}
For every $\lambda\in (0,\infty)$, we have
\begin{equation}
   \label{eq:lim-u/c}
\lim_{t\rightarrow \infty } \frac{u(\lambda,t)}{c(t)}=\expp{ -\alpha
  c^{-1} (\lambda)},
\end{equation}
and
there exists $\kappa_*\in(0,\infty)$ 
such that
\begin{align}
\lim_{t\to\infty}c(t) {\rm e}^{\alpha t}=\kappa_*.\label{eq:asymp-c}
\end{align}
\end{lem}
We also compute some integral of $\tilde \psi'$. 
\begin{prop}
\label{prop:cv-int-psi-u}
The followings hold for every $0\leq t<\infty$:
\begin{align}
\int_t^\infty\tilde{\psi}'(u(\lambda,s))ds
=&\log\left(\frac{\psi(u(\lambda,t))
\expp{\alpha t+\alpha c^{-1}(\lambda)}}{\kappa_*\alpha}\right),\quad
\lambda>0, 
\label{eq:int-tilde-psi0}\\
\int_t^\infty\tilde{\psi}'(c(s))ds
=&\log\left(\frac{\psi(c(t))
{\rm e}^{\alpha t}}{\kappa_*\alpha}\right),\label{eq:int-tilde-psi1}
\end{align}
where the constant $\kappa_*$ is defined in Lemma~\ref{lem:AL}.
\end{prop}
\begin{proof}
We deduce from \reff{eq:aux6}, 
\reff{eq:lim-u/c} and \reff{eq:asymp-c} that:
\[
\lim_{T\to\infty}\psi(u(\lambda,T))\expp{\alpha
  T}=\lim_{T\to\infty}\frac{\psi(u(\lambda,T))}{u(\lambda,T)}
\frac{u(\lambda,T)}{c(T)}c(T)\expp{\alpha T}
=\alpha\;\expp{-\alpha c^{-1}(\lambda)}\kappa_*, 
\]
and \reff{eq:int-tilde-psi0} follows by letting
$T\longrightarrow\infty$ for both sides of (\ref{eq:aux6}). Then, let
$\lambda$ goes to infinity in \reff{eq:int-tilde-psi0}  to get 
\reff{eq:int-tilde-psi1} and use the monotone convergence theorem.
\end{proof}

As a consequence of \reff{eq:int-tilde-psi0} with $t=0$ and Lemma
\ref{lem:Laplace-Z}, we get the following Corollary.

\begin{cor}
   \label{cor:Laplace-Z2}
For all $t\in \R$ and $\lambda\geq 0$, the Laplace transform of
$Z_t$ is given by: 
\begin{equation}
   \label{eq:Laplace-Z2}
\E\left[\expp{-\lambda Z_t}\right]=\exp\left(-\int_0^\infty
ds\;  \tilde \psi'(u(\lambda,s))  \;
\right)= 
\frac{\expp{-\alpha c^{-1}(\lambda)}\kappa_*\alpha}{\psi(\lambda)}.
\end{equation}
\end{cor}

Eventually, we check that $Z$ is non-zero. Recall notations from Section
\ref{sec:ppm}. Let $\zeta_i=\inf\{t>0; Y^i_t=0\}$ be the duration of the
family or clan $Y^i$ and $t_i+\zeta_i$ its extinction time, with
$i$ in $I$, $J_1$ or $J_2$.
\begin{prop}
   \label{prop:Z=0}
We have
\[
\P\left(\sum_{i\in  \ci}\ind_{(t_i,t_i+\zeta_i)}(t)>0,\;\forall t\in
  \R\right)=1,
\]
In  particular, we have $\P(\exists t\in \R;
Z_t=0)=0$.
\end{prop}

   For $-\infty <a<b<+\infty $, we will consider in the forthcoming
   proof 
\begin{equation}
\label{def:NA,B} 
N_{a,b}=\sum_{i\in \ci}\ind_{\{t_i<a; b<t_i+\zeta_i\}},
\end{equation}
the number  of clans born before  $a$ and still  alive at time
$b$. Notice $N_{a,b}$ is a Poisson random variable with parameter
\begin{align}
\nonumber
\Lambda(b-a):
&=\int dr  \mu(dY)\; \ind_{(-\infty,a)}(r)\ind_{\{\zeta+r>b\}} \\
\nonumber
&=\int_{b-a}^\infty dr\;\tilde{\psi}'(c(r))\\
&=\log\left(\frac{\psi(c(b-a))\expp{\alpha(b-a)}}{\kappa_*\alpha}\right),
\label{eq:lambdaAB} 
\end{align}
where we used \reff{eq:def-mu} the definition of $\mu$ for the first
equality and \reff{eq:int-tilde-psi1} for the last equality. 

\begin{proof}
Observe that no clan surviving at time
$t\in (a,b)$ implies that there are no  clan
surviving on any non-degenerate interval containing $t$. Hence, for any
$n\geq 1$, we have:
\[
\left\{\exists t\in (a,b),\;\sum_{i\in \ci}\ind_{(t_i,t_i+\zeta_i)}(t)=0
\right\} \subset \bigcup_{j=1}^n\left\{N_{u_{j-1},u_j}=0\right\}
\cup \bigcup_{j=1}^{n+1}\left\{N_{v_{j-1},v_j}=0\right\}
\]
where $u_j=a+j(b-a)/n$ and $v_j=a+(2j-1)(b-a)/2n$. 
Notice that  $N_{u_{j-1},u_j}$ and $N_{v_{j-1},v_j}$ are Poisson random
variables  with parameter $\theta_n=\Lambda((b-a)/n)$. We deduce that 
\begin{equation}
\label{eq:aux01}
\P\left(\exists t\in (a,b),\;\sum_{i\in \ci}\ind_{(t_i,t_i+\zeta_i)}(t)=
  0\right)\leq (2n+1)\expp{-\theta_n}.
\end{equation}
Therefore the first part of the Proposition will be proved as soon as
$\lim_{n\rightarrow +\infty } n \exp(-\theta_n)=0$ which, thanks to
formula \reff{eq:lambdaAB}, will be implied by $\displaystyle
\lim_{t\rightarrow 0} t\psi(c(t))=+\infty $ and thus by
\begin{equation}
   \label{eq:lim-Lambda}
\lim_{\lambda\rightarrow +\infty } \int_\lambda^{+\infty }
\frac{dr}{\psi(r)} \psi(\lambda)=+\infty .
\end{equation}
Hypothesis on $\beta$ and $\pi$ imply there exists a constant $c_0>0$
such that 
\[
\alpha \lambda \leq \psi(\lambda)\leq  c_0\lambda^2
\quad\text{and}\quad
\lim_{\lambda\rightarrow+\infty } \psi(\lambda)/\lambda=+\infty .
\]
Therefore \reff{eq:lim-Lambda} is in force. 

The second part of the Proposition is clear by definition of $\zeta_i$
and representation \reff{eq:def-Z2}. 
\end{proof}

\section{TMRCA and populations sizes}
\label{sec:TMRCA}
We  consider the coalescence  of the  genealogy at  a fixed  time $t_0$.
Thanks  to stationarity, we  may assume  that $t_0=0$  and we  write $Z$
instead of  $Z_0$.  There are  infinitely many clans
contributing to the population at  time $0$. The Poisson random variable
introduced  in   \reff{def:NA,B},  with   $b=0$,  gives  the   number  of
clans born before $a$ and  still alive at time $0$. Notice its
parameter is  finite, see \reff{eq:lambdaAB}. Therefore,  there are only
finitely  many  clans  born  before  $a$  and  alive  at  time
$0$.  In particular, this
implies that there is one  unique oldest clan  alive
at time $0$. We denote by $-A$ the birth time of this unique oldest
clan  at time  $0$:
\[
A=-\inf\{t_i\leq 0;Y^i_{-t_i}>0, i\in \ci\}.
\]
We set  $Z^O$  the population size of this  clan
at time $0$:
\[
Z^O:= Y^i_{-t_i},\quad \quad\mbox{ if }A=-t_i.
\]
The time $A$ is also the time to the most recent common ancestor (TMRCA)
of the population at time $0$. 
The size of all the clans alive at time $0$ with birth time in $(-A, 0)$
is given by  
\[
Z^I:=Z-Z^O.
\]
We are also interested in the size of the population just before
the most recent common ancestor (MRCA): 
\[
Z^A:=Z_{(-A)-}=\sum_{i\in \ci} Y^i_{(-A-t_i)}\ind_{\{t_i<-A\}}. 
\]

\begin{theo}
\label{theo:(Z,ZA,ZI,ZO)-dsitrib}
The joint distribution of $(A,Z^A,Z^I,Z^O)$ is characterized by the following: for $\lambda,\gamma,\eta\geq 0$ and $t\geq 0$,
\begin{multline}
\label{eq:(A,ZA,ZI,ZO)}
\E\left[\expp{-\lambda Z^A -\gamma Z^I-\eta Z^O}; A\in dt\right]\\
=dt\left(\tilde{\psi}'(c(t))-\tilde{\psi}'(u(\eta,t))\right)
\times\exp\left(
-\int_0^tds\; \tilde{\psi}'(u(\gamma,s))-\int_0^\infty ds \;
\tilde{\psi}'(u(\lambda+c(t),s))\right). 
\end{multline}
\end{theo}
\begin{proof}
Given $f$ a non-negative Borel measurable function defined on $\R$,  we have
\begin{multline*}
   \mathbb  E\left[\expp{-\lambda Z^A-\gamma Z^I-\eta
       Z^O}f(A)\right]\\
\begin{aligned}
&=\mathbb  E\Bigg[\sum_{j\in \ci}
\exp \left(-\lambda \sum_{i\in \ci, t_i< t_j}Y^i_{(t_j-t_i)}-\gamma
\sum_{i\in \ci, t_i>t_j}Y^i_{-t_i}-\eta Y^j_{-t_j}\right)\\
&\hspace{4cm} f(-t_j)\; \ind_{\left\{
Y^j_{-t_j}>0,\sum_{i\in \ci,
  t_i<t_j}\ind_{\{Y^i_{-t_i}>0\}}=0\right\}}\Bigg]\\  
&=\int_0^\infty dt\; \mu\left(\expp{-\eta Y_{t}};Y_{t}>0\right)f(t)\; \mathbb 
E\left[\exp\left(-\gamma\sum_{i\in \ci, t_i>-t} Y^i_{-t_i}\right)\right]\\
&\hspace{4cm} \lim_{K\to\infty}\mathbb  E\left[\exp
  \left(-\lambda\sum_{t_i<-t}\left(Y^i_{(-t-t_i)} 
+K\ind_{\left\{Y^i_{-t_i}>0\right\}}\right)\right)\right],
\end{aligned}
\end{multline*}
where we used that Poisson point measures over disjoint sets are
independent. We have: 
\[
\mu\left(\expp{-\eta   Y_{t}};Y_{t}>0\right)
=\mu\left(\ind_{\{Y_{t}>0\}}-\left(1- \expp{-\eta Y_{t}}\right)\right)
=\tilde{\psi}'(c(t))-\tilde{\psi}'(u(\eta,t)).
\]
Using Lemma \ref{lem:EFfam}, we get:
\[
\mathbb  E\left[\exp\left(-\gamma\sum_{i\in \ci,
      t_i>-t}Y^i_{-t_i}\right)\right]
=\exp\left(-\int_0^{t}ds\; \tilde{\psi}'(u(\gamma,s))\right). 
\]
We also have:
\begin{multline*}
\lim_{K\to \infty}\mathbb  E\left[\exp \left(-\lambda\sum_{i\in \ci,
      t_i<-t}\left(Y^i_{(-t-t_i)} 
+K\ind_{\left\{Y^i_{-t_i}>0\right\}}\right)\right)\right]\\
\begin{aligned}
=&\exp\left(-\int ds \; \ind_{\{s>0\}}\; \mu\left(1-\expp{-\lambda
      Y_s}\ind_{\left\{Y_{s+t}=0\right\}}\right)\right)\\ 
=&\exp\left(-\int ds\; \ind_{\{s>0\}}\; \mu\left(1-\expp{-\lambda
      Y_s}\rP_{Y_s} \left(Y_{t}=0\right)\right)\right)\\ 
=&\exp\left(-\int ds \; \ind_{\{s>0\}}\; \mu\left(1-\expp{-(\lambda+c(t))Y
      _s}\right)\right)\\ 
=&\exp\left(-\int_0^\infty ds\; \tilde{\psi}'(u(\lambda+c(t),s))\right),
\end{aligned}   
\end{multline*}
where we used exponential formulas for Poisson point measure in the
first equality and  the Markov property of $Y$ for the second
equality.  Putting things together, we then get \reff{eq:(A,ZA,ZI,ZO)}.
\end{proof}

It is then easy to derive the distribution of the TMRCA $A$. 
\begin{cor}
The distribution function of $A$ is given by
\[
\P(A\leq t)=\E[\expp{-c(t) Z}]=\exp\left(-\int_t^\infty ds\; 
  \tilde \psi'(c(s))\right),
\]
and $A$ has density, $f_A$, with respect to the Lebesgue measure given by:
\begin{equation}
   \label{eq:fA}
f_A(t)=\tilde{\psi}'(c(t))\exp\left(-\int_t^\infty ds\; 
  \tilde \psi'(c(s))\right)\ind_{\{t>0\}}=
\frac{\tilde{\psi}'(c(t))}{\psi(c(t))}\expp{-\alpha
  t}\kappa_*\alpha\ind_{\{t>0\}}.
\end{equation}
\end{cor}
\begin{proof}
This is a direct consequence of Theorem
\ref{theo:(Z,ZA,ZI,ZO)-dsitrib} and \reff{eq:uc=c}. 
Use Lemma~\ref{lem:Laplace-Z} to get \reff{eq:fA}.
\end{proof}

The next result is a direct consequence of Theorem
\ref{theo:(Z,ZA,ZI,ZO)-dsitrib}. 
\begin{cor}
\label{cor:Indep-AZAZ}
Conditionally on $A$, the three random variables $Z^I,Z^A$ and $Z^O$ are
independent. 
\end{cor}

We can also give the mean of the population size just before
the most recent common ancestor (MRCA) (to be compared to the mean size
of the current population given by \reff{eq:mom1Z0}). 
\begin{cor}
   \label{cor:moment1ZA}
Let $t>0$. We have 
\begin{equation}
   \label{eq:LapZA}
\E\left[\expp{-\lambda Z^A}|A=t\right]=
\frac{\E\left[\expp{-(\lambda+c(t))Z}\right]}{\E\left[\expp{-c(t)Z}\right]}
\quad\text{and }\quad 
 \E[Z^A|A=t]=\frac{\tilde
  \psi'(c(t))}{\psi(c(t))}.
\end{equation}
\end{cor}
\begin{proof}
   This is  a direct consequence of Theorem
\ref{theo:(Z,ZA,ZI,ZO)-dsitrib} and of \reff{eq:mom1Z}. 
\end{proof}
We deduce from \reff{eq:LapZA} that the distribution of
$Z^A$ conditionally on $\{A=t\}$ converges, as $t$ goes to infinity, to
the distribution of $Z$. 

\bigskip

As another application of Theorem~\ref{theo:(Z,ZA,ZI,ZO)-dsitrib}, we
get that the population just before the MRCA, $Z^A$, is stochastically
smaller than  the current population, $Z$. Note that strong inequality,
namely inequality in the almost-surely sense, does not hold in general
(see Section~\ref{sect:ex}).

\begin{prop}
\label{theo:bottleneck}
We have $\P(Z^A\leq z|A=t)\geq \P(Z\leq z)$ for all $z\geq 0$ and $t\geq
0$. Hence, the population size $Z^A$ is stochastically smaller than $Z$:
$\P(Z^A\leq z)\geq \P(Z\leq z)$ for all $z\geq 0$. In particular, we
have
\[
\mathbb  E[Z^A|A]\leq \mathbb  E[Z]\quad \text{a.s.}
\]
\end{prop}
\begin{proof}
  The   first  equality   of  \reff{eq:LapZA}   implies  that   for  any
  non-negative measurable function $F$ defined on $\R$,
\[
\E\left[F(Z^A)|A=t\right]=
\frac{\E\left[F(Z)\expp{-c(t)Z}\right]}{\E\left[\expp{-c(t)Z}\right]}.
\]
Note that $\expp{-c(t)Z}-
\E\left[\expp{-c(t)Z}\right]$ is non-negative for $Z$ less than
$\inv{-c(t)} \log \left( \E\left[\expp{-c(t)Z}\right]\right)$ and
non-positive otherwise, and that  $\lim_{z\to \infty}
\E[\expp{-c(t)Z};Z\leq z]-\mathbb  E[\expp{-c(t)Z}]\mathbb  P(Z\leq z)=0$. We
deduce that:
\[
\P(Z^A\leq z|A=t)=
\frac{\E\left[\expp{-c(t)Z}; Z\leq
    z\right]}{\E\left[\expp{-c(t)Z}\right]}
\geq \P(Z\leq z).
\]
For the last assertion, recall that for any non-negative random variable,
we have $\mathbb  E[X]=\int_0^\infty \mathbb  P(X> x)dx$.
\end{proof}

\begin{rem}
\label{rem:ZA+}
  Instead of considering $Z^A$, the size of the population just before
  the MRCA, we could consider the size of the population at the MRCA,
  $Z^A_+$, which is formally given by 
\[
Z^A_{+}=Z^A+ \sum_{i\in I} Y^i_0\; \ind_{\{t_i=-A\}}.
\]
Notice we don't take into account  the contribution of $i\in J_2$ as for
those indices  we have $Y^i_0=0$. (In  particular if $\pi=0$,  then $Z$ is
continuous and $Z^A=Z^A_+$.) Similar  computations as those in the proof
of Theorem \ref{theo:(Z,ZA,ZI,ZO)-dsitrib} yield: for $\lambda,t>0$
\[
\E[\expp{-\lambda Z^A_+}|A=t]
=
\E[\expp{-\lambda Z^A}|A=t] 
\frac{\psi'(\lambda+c(t)) -\psi'(\lambda)}{\psi'(c(t))-\psi'(0)}. 
\]
If $\psi''(0)=+\infty $, then we get that $\displaystyle
\lim_{t\rightarrow +\infty } \E[\expp{-\lambda Z^A_+}|A=t]=0$.  Thus,
conditionally on $\{A=t\}$, for $t$ large, we have that $Z^A_+$ is
likely to be very large. (Intuitively, a clan is born at time $-t$ which
has survive up to time $0$; and if $t$ is large, it is very likely to
have a large initial size.) Therefore, $Z^A_+$ is not  stochastically
smaller than $Z$ in the general case. 
\end{rem}
\bigskip

We may also consider the  TMRCA of the immortal individual and individuals
taken independently and uniformly among the current population living at
time $t$.  Let $J^n_t\subset  \ci$ be  the indices of  the clans  of the
randomly  chosen $n$  individuals  alive at  time  $t$. (One  individual
chosen at random  in the population at time $t$ belongs  to the clan, $i$
with  probability $Y^i_{t-t_i}/Z_t$.)  Notice that  $\Card(J^n_t)\leq n$.
The TMRCA  for the $n$  individuals alive at  time $t$ and  the immortal
individual is given by:
\[
A^n_t:=-\inf\{ t_i; i\in J^n_t, i\in\ci\}.
\]
Because of the stationarity, we shall focus on $t=0$ and write $A^n$ for
$A^n_t$.  The  joint law  of $Z$  and $A^n$ can  be characterized  by the
following result.

\begin{theo}
\label{theo:Z,An0}
For any $n\geq 1$ and any $\lambda,T\geq 0$, we have
\[
\mathbb  E\left[Z^{n}\expp{-\lambda Z}\ind_{\left\{A^n\leq T\right\}}\right]
=\frac{\expp{-\alpha c^{-1}(\lambda)}\kappa_*\alpha}{\psi(u(\lambda,T))} 
(-1)^{n}\frac{\partial^n
}{\partial^{n}\eta}\left(\frac{\psi(u(\lambda+\eta,T))}
{\psi(\lambda+\eta)}\right)\Big|_{\eta=0}.
\]
\end{theo}

\begin{proof}
By definition, we have:
\begin{multline*}
\mathbb  E\left[Z^n\expp{-\lambda Z}\ind_{\{A^n\leq T\}}\right]\\
\begin{aligned}
&=\mathbb 
E\left[Z^n\sum_{i_1,\cdots,i_n}\frac{Y^{i_1}_{-t_{i_1}}}{Z}\cdots 
\frac{Y^{i_n}
_{-t_{i_n}}}{Z}\prod_{k=1}^n\ind_{\{-t_i\leq T\}}\expp{-\lambda Z}\right]\\
&=\mathbb  E\left[\left(\int \cn_3(ds,dY)\;Y_{-s}\ind_{\{-s\leq
      T\}}\right)^n\exp \left(-\lambda\int \cn_3(ds,dY)\;Y_{-s}\right)\right]\\ 
&=(-1)^n\frac{\partial^n}{\partial^n\eta}\mathbb  E\left[
\exp\left(-\int \cn_3(ds,dY)\;\left(\eta Y_{-s} \ind_{\{-s\leq
      T\}}+\lambda  Y_{-s}\right)\right)\right]\Bigg|_{\eta=0} \\
 &=(-1)^n\frac{\partial^n}{\partial^n\eta}\exp\left(
 -\int_T^\infty
 ds\; \tilde{\psi}'(u(\lambda,s))-\int_0^Tds\; \tilde{\psi}'
 (u(\lambda+\eta,s))\right)\Bigg|_{\eta=0},
\end{aligned}
\end{multline*}
where $\cn_3$ in the second equality is defined by \reff{eq:defn3}. 
The result  then follows from  \reff{eq:aux6} and \reff{eq:int-tilde-psi0}.
\end{proof}

\begin{rem}
Following almost the same lines as the proof of
Theorem~\ref{theo:Z,An0}, one can characterize explicitly the joint
distribution of $\left\{\left(Z_{r_j},A^{n_j}_{r_j}\right);1\leq j\leq
  m\right\}$ for any $m, n_1,\cdots,n_m \in \mathbb  N^*$  and
$-\infty<r_1<r_2<\cdots<r_m<\infty$. 
\end{rem}

\section{Number of old families}\label{sec:nof}
We now  consider the number  families in the  oldest clan alive  at time
$0$. This correspond  to the number of individuals  involved in the last
coalescent  event of the  genealogical tree.  To this  end, we  take the
representation \reff{eq:def-Z} for $Z$.

\begin{defi}
The number of oldest families alive at time $0$ (excluding the immortal particle) is defined by:
\begin{equation}
   \label{eq:defNA}
N^A=\sum_{j\in \cj}\ind_{\{A=-t_j, \; Y^j_{-t_j}>0\}}
=\sum_{j\in \cj}\ind_{\{A=-t_j, \; \zeta_j>-t_j\}}.
\end{equation}
\end{defi}

We have  $N^A\geq 1$. In the  particular case $\pi=0$  and $\beta>0$, we
have $\cj=J_2$ and $N^A=1$.

\bigskip

The following proposition give the joint law of $A$, $N^A$ and $Z$.  

\begin{prop}
\label{prop:NA}
We have for $a\in [0,1]$, $\lambda\geq 0$, $t\geq 0$,
\[
\E\left[a^{ N^A}\expp{-\lambda Z}|A=t\right]=\frac{\psi'(c(t)) -
  \psi'((1-a)c(t)+ au(\lambda,t))} {\tilde \psi'(c(t))}\expp{-\int_0^t
  \tilde \psi'(u(\lambda,r))\; dr}.
\]
and
\[
\E\left[a^{ N^A}|A=t\right]=\frac{\psi'(c(t)) -
  \psi'((1-a)c(t))} {\tilde \psi'(c(t))}=1 - \frac{\tilde
  \psi'\Big((1-a)c(t)\Big)}{\tilde \psi'\Big(c(t)\Big)}.
\]
\end{prop}
\begin{proof}
Recall notations from Section \ref{sec:ppm}. For $i\in \ci$, we set
$J^*_i=J_{1,i} $ if $i\in I$ and $J^*_i=\{i\}$ if $i\in  J_2$. 
Given any non-negative function $f$, we have, using \reff{eq:def-Z} and
\reff{eq:def-Z2}: 
\begin{multline*}
   \E\left[a^{N^A}\expp{-\lambda Z}f(A)\right]\\
\begin{aligned}
&=\E\left[\expp{-\lambda \sum_{k\in \ci} Y_{-t_{k}}^{k}} 
\sum_{i\in \ci} a^{\sum_{j\in J^*_i}\ind_{\{\zeta_j>-t_i\}}}
   f(-t_i)\;\ind_{\{ Y^i_{-t_i}\neq 0\}}\ind_{\big\{\sum_{k'\in
     \ci,t_{k'}<t_i}\ind_{\{Y^{k'}_{-t_i} >0\}}=0 \big\}} \right]\\ 
&=\int_0^\infty ds\; f(s)\; \mathbb  E\left[\expp{-\lambda \sum_{k\in \ci} Y^k
    _{-t_k} \ind_{\{t_k>-s\}}}\right]\; 
 \P\left(\sum_{k\in \ci} \ind_{\{t_k<-s,\; Y^k_{s}>0 \}}
    =0\right) \\  
  &\hspace{.4cm}\times \left(2\beta \N\left[
a\expp{-\lambda Y_s }\ind_{\{Y_s >0\}} 
\right]+\int_{(0,+\infty )} \!\!\!\!\!\!\!\!\! \!\ell \pi(d\ell)\;
\E_\ell\left[a^{\sum_{j\in 
      J_3}\ind_{\{Y^j_s>0\}}}\expp{-\lambda \sum_{j\in
      J_3}Y^j_s}\ind_{\{\sum_{j\in
      J_3}Y^j_s>0\}} 
\right]\right) ,
\end{aligned}
\end{multline*}
where $\sum_{j\in J_3} \delta_{Y^j}(dY)$ is under $\E_\ell$ a Poisson point
measure with intensity $\ell\N[dY]$. We have
\[
   \mathbb  E\left[\expp{-\lambda \sum_{k\in \ci} Y^k
    _{-t_k} \ind_{\{t_k>-s\}}}\right]\; 
 \P\left(\sum_{k\in \ci} \ind_{\{t_k<-s,\; Y^k_{s}>0 \}}
    =0\right) 
= \expp{-\int_0^{s}dr\; \tilde{\psi}'(u(\lambda,r))-\int_{s}^\infty
    dr \; \tilde{\psi}'(c(r))}.
\]
We also have 
\[
\N\left[
\expp{-\lambda Y_s }\ind_{\{Y_s >0\}} 
\right]= \N[Y_s>0] - \N[1-\expp{-\lambda Y_s}]=c(s) - u(\lambda,s). 
\]
and 
\begin{multline*}
  \E_\ell\left[a^{\sum_{j\in 
      J_3}\ind_{\{Y^j_s>0\}}}\expp{-\lambda \sum_{j\in
      J_3}Y^j_s}\ind_{\{\sum_{j\in
      J_3}Y^j_s>0\}} 
\right]\\
\begin{aligned}
   &= \E_\ell\left[a^{\sum_{j\in 
      J_3}\ind_{\{Y^j_s>0\}}}\expp{-\lambda \sum_{j\in
      J_3}Y^j_s}
\right] - \P_\ell\left(\sum_{j\in      J_3}Y^j_s=0\right) \\
&=\exp\left(-  \ell\N[(1 -a\expp{-\lambda Y_s})\ind_{\{Y_s>0\}}]\right) - \exp\left(-
  \ell \N[Y_s>0]\right)\\
&=\exp\left(-  \ell \N[Y_s>0] + \ell a\N[\expp{-\lambda
    Y_s}]\ind_{\{Y_s>0\}}]\right)  - \exp\left(-
  \ell \N[Y_s>0] \right)\\
&=\exp\left(-  \ell \big((1-a) c(s)- au(\lambda,s)\big) \right) -
\exp\left(-   \ell c(s)\right).
 \end{aligned} 
\end{multline*}
Thus, we get:
\begin{multline*}
   2\beta \N\left[
a\expp{-\lambda Y_s }\ind_{\{Y_s >0\}} 
\right]+\int_{(0,+\infty )} \!\!\!\!\!\!\!\!\! \!\ell \pi(d\ell)\;
\E_\ell\left[a^{\sum_{j\in 
      J_3}\ind_{\{Y^j_s>0\}}}\expp{-\lambda \sum_{j\in
      J_3}Y^j_s}\ind_{\{\sum_{j\in
      J_3}Y^j_s>0\}} 
\right]\\
= \psi'(c(s))-\psi'((1-a)c(s)+au(\lambda,s)).
\end{multline*}
Putting things together, we obtain:
\begin{multline*}
     \E\left[a^{N^A}\expp{-\lambda Z}f(A)\right]\\
=\int_0^\infty ds\; f(s)\;\expp{-\int_0^{s}dr\;
  \tilde{\psi}'(u(\lambda,r))-\int_{s}^\infty     dr \;
    \tilde{\psi}'(c(r))}\left[\psi'(c(s))-\psi'((1-a)c(s)+au(\lambda,s))
  \right].   
\end{multline*}
Then, use \reff{eq:fA}  for the density of $A$ to get the result. 
\end{proof}

\begin{cor}
   \label{cor:ENA}
We have:
\begin{equation}
   \label{eq:PNA=}
\P(N^A=n|A=t)=(-1)^{n+1}\frac{c(t)^n \psi^{(n+1)} (c(t))}{n!\; \tilde
  \psi'(c(t))},\quad n\in \mathbb  N^*. 
\end{equation}
Suppose that $\psi''(0+)<\infty$ (that is $\E[Z]<+\infty $). Then, we have
\[
\E[N^A|A=t]= \psi''(0)\frac{c(t)}{\tilde \psi'(c(t))}.
\]
Furthermore the function $t\longmapsto \E[N^A|A=t]$ is non-increasing.
\end{cor}
\begin{proof}
The first two assertions are straightforward consequences of Proposition
\ref{prop:NA}. To get the monotonicity of $t\longmapsto \E[N^A|A=t]$, we
simply notice that both $t\longmapsto c(t)$ and 
\[
x\longmapsto \frac{\tilde \psi'(x)}{x}=2\beta+ \int _0^\infty
\pi(d\ell)\ell\; \frac{1-\expp{-x\ell}}{x}
\]
are non-increasing.
\end{proof}

\begin{rem}
\label{rem:NA+i}
Suppose that $\psi''(0+)<\infty$. We deduce from 
\reff{eq:PNA=} that 
\[
\lim_{t\rightarrow +\infty } \P(N^A=1|A=t)=1.
\]
Thus, the distribution of  $N^A$ conditionally on $\{A=t\}$ converges as
$t$ goes to infinity to $1$. So roughly speaking $N^A$ is  likely to
be equal to 1 if the TMRCA (or age of the oldest clan alive) is
large. Notice that if $\psi''(0+)=+\infty $, this result may be false
(see the next Remark). 
\end{rem}

\begin{rem}
\label{rem:stable-NA}
Let us consider the stable cases, $\psi(\lambda)=\alpha \lambda+
c_0\lambda^{1+\alpha_0}$, with $c_0>0$ and 
$\alpha_0\in (0,1]$.  We
deduce from Corollary  \ref{cor:ENA} that 
\[
\E[a^{N^A}|A=t]= 1-(1-a)^{\alpha_0}.  
\]
In particular $N^A$ is independent of $A$. 
The case $\alpha_0=1$ correspond to  the quadratic branching
mechanism and we get that a.s. $N^A=1$. For $\alpha_0\in (0,1)$, we
deduce from  \reff{eq:PNA=} that: for $n\in \N^*$
\[
\P(N^A=n|A=t)=\inv{n!}\alpha_0 \prod_{k=1}^{n-1} (k-\alpha_0).
\]
For $\alpha_0\in (0,1)$, we have $\psi''(0+)=+\infty $ and the result of
Remark \ref{rem:NA+i} does not hold.
\end{rem}

\section{Asymptotics for the number of ancestors}
\label{sec:ancestors}

The number $N_{-s,0}$ defined by \reff{def:NA,B} of clans born before
time $-s$ and alive at time $0$ is non-decreasing and is distributed as
a Poisson random variable with parameter 
$\Lambda(s)$ given by \reff{eq:lambdaAB}. As $\Lambda(s)$ goes to
infinity as $s$ goes down to $0$, we deduce that $N_{-s,0}$ tends to
infinity almost surely as $s\downarrow 0+$. A natural question is then
how fast the numbers $N_{-s,0}$ tend to infinity. 
It follows from the definition of the Poisson random measure $\cn_3$ in
\reff{eq:defn3} that $\{N_{-\Lambda^{-1}(s),0};s\geq 0\}$ is Poisson
process with parameter $1$, and by the strong law of large numbers for
L\'evy processes (see \cite{b:pl}), we deduce that 
\[
\lim_{s\downarrow 0+}\frac{N_{-s,0}}{\Lambda(s)}=1\quad \mbox{ almost surely}.
\]

One  can also  ask  how  fast the  number  $M_s$ of
ancestors  at time $-s$  of the  current population  living at  time $0$
tends to  infinity. To  answer this question,  we need to  introduce the
genealogy of  the families. Notice the  genealogy of a CB  is a richer
structure than the CB itself.

\subsection{Genealogy of CB}

We recall here the construction  of the Lévy continuum random tree (CRT)
introduced  in \cite{lglj:bplpep,lglj:bplplfss}  and developed  later in
\cite{dlg:rtlpsbp} for critical or sub-critical branching mechanism. The
results of  this section  are mainly extracted  from \cite{dlg:rtlpsbp},
except for the next subsection which is extracted from \cite{lg:rrt}.

\subsubsection{Real trees and their coding by a continuous function} 

Let us first recall the definition of real trees.

\begin{defi}
A metric space $(\ct,d)$ is a real tree if the following two
properties hold for every $v_1,v_2\in\ct$.
\begin{itemize}
\item (Unique geodesic.)
There is a unique isometric map $f_{v_1,v_2}$
  from $[0,d(v_1,v_2)]$ into $\ct$ such that
$$f_{v_1,v_2}(0)=v_1\qquad\mbox{and}\qquad
  f_{v_1,v_2}(d(v_1,v_2))=v_2.$$
\item (No loop.)
If $q$ is a continuous injective map from $[0,1]$ into
  $\ct$ such that $q(0)=v_1$ and $q(1)=v_2$, we have
$$q([0,1])=f_{v_1,v_2}([0,d(v_1,v_2)]).$$
\end{itemize}
A rooted real tree is a real tree $(\ct,d)$ with a distinguished
vertex $v_\emptyset$ called the root.
\end{defi}

Let  $(\ct,d)$  be  a  rooted  real  tree.  The  range  of  the  mapping
$f_{v_1,v_2}$ is denoted  by $\lb v_1,v_2\rb$ (this is  the line between
$v_1$ and $v_2$ in the tree). In particular, for every vertex $v\in\ct$,
$\lb v_\emptyset,v\rb$ is  the path going from the root  to $v$ which we
call the  ancestral line of  vertex $v$. More  generally, we say  that a
vertex   $v$   is  an   ancestor   of   a   vertex  $v'$   if   $v\in\lb
v_\emptyset,v'\rb$. If $(v_k\in  K)$ is a set of  vertex of $\ct$, there
is a unique  $a\in \ct$ such that $\lb  v_\emptyset, a\rb=\bigcap _{k\in
  K}  \lb  v_\emptyset,v_k\rb$.  We  call  $a$  the  most recent  common
ancestor of $(v_k\in  K)$.  A leaf is a vertex which  is the ancestor of
itself only. We  say that $d(\emptyset,v)$ is the  level (or generation)
of the vertex $v$.


We now recall the coding of a compact real tree by a continuous
function  $g\,:\, [0,+\infty)\longrightarrow
  [0,+\infty)$ with compact support and such that $g(0)=0$. We also
    assume that $g$ is not identically 0. For every $0\le s\le t$, we set
\[
m_g(s,t)=\inf_{u\in[s,t]}g(u) 
\quad\text{and}\quad
d_g(s,t)=g(s)+g(t)-2m_g(s,t).
\]
We then introduce the equivalence relation $s\sim t$ if and only if
$d_g(s,t)=0$. Let $\ct_g$ be the quotient space $[0,+\infty)/\sim$. It
  is easy to check that $d_g$ induces a distance on $\ct_g$. Moreover,
  $(\ct_g,d_g)$ is a compact real tree (see \cite{dlg:pfalt}, Theorem
  2.1). We say that $g$ is the height process of the tree $\ct_g$.

For instance,
when $g$ is a normalized Brownian excursion, the associated real tree is
Aldous' CRT \cite{a:crt3}. 

\subsubsection{The underlying Lévy process}\label{subsec:levy}
We present now how to define a height
process that codes a random real trees describing the genealogy of a
CB using a  L\'evy process with Laplace exponent given by the
branching mechanism $\psi$. We shall consider only the case of the sub-critical
branching mechanism $\psi$ given by \reff{eq:def-psi}. 

Let  $X=(X_t,t\ge 0)$  be a  $\R$-valued Lévy  process with  no negative
jumps,  starting from  0  and  with Laplace  exponent  $\psi$ under  the
probability measure $\bP$ (and $\bE$ the corresponding expectation): for
$\lambda\ge       0$,       $\displaystyle       \bE\left[\expp{-\lambda
    X_t}\right]=\expp{t\psi(\lambda)}$.  Since  we assume that $\beta>0$
or $\pi((0,1))=+\infty $, we get that a.s. $X$ is of infinite variation.

We introduce some processes related to $X$.  
Let $I=(I_t,t\ge 0)$
be  the infimum  process of  $X$, $I_t=\inf_{0\le  s\le t}X_s$,  and let
$S=(S_t,t\ge 0)$  be the supremum process,  $S_t=\sup_{0\le s\le t}X_s$.
We will  also consider for every $0\le  s\le t$ the infimum  of $X$ over
$[s,t]$:
\[
I_t^s=\inf_{s\le r\le t}X_r.
\]
The point  0 is regular  for the Markov  process $X-I$, and $-I$  is the
local time of $X-I$ at 0  (see \cite{b:pl}, chap. VII). Let $\bN$ be
the associated excursion  measure of the process $X-I$  away from 0. Let
$\sigma=\inf\{t>0; X_t-I_t=0\}$ be the  length of the excursion of $X-I$
under $\bN$. We have $X_0=I_0=0$ $\bN$-a.e.

Since $X$ is of infinite variation, 0 is also regular for the Markov 
process $S-X$. The local time, $L=(L_t, t\geq 0)$, of $S-X$ at 0 will be
normalized so that 
\[
\bE[\expp{-\lambda S_{L^{-1}_t}}]= \expp{- t \psi(\lambda)/\lambda},
\]
where $L^{-1}_t=\inf\{ s\geq 0; L_s\geq t\}$ (see also \cite{b:pl}
Theorem VII.4 (ii)).

\subsubsection{The height process and the Lévy CRT}
For each $t\geq 0$, we consider the reversed process at time $t$,
$\hat X^{(t)}=(\hat X^{(t)}_s,0\le s\le t)$ by:
\[
\hat X^{(t)}_s=
X_t-X_{(t-s)-} \quad \mbox{if}\quad  0\le s<t,
\]
and $\hat  X^{(t)}_t=X_t$. The two processes  $(\hat X^{(t)}_s,0\le s\le
t)$ and $(X_s,0\le s\le t)$ have the same law. Let $\hat S^{(t)}$ be the
supremum process of $\hat X^{(t)}$ and $\hat L^ {(t)}$ be the local time
at $0$ of  $\hat S^{(t)} - \hat X^{(t)}$ with  the same normalization as
$L$. As assumption  \reff{eq:A1} is in force, there  exists a continuous
modification $H=(H_t,t\ge  0)$ of  the process $(\hat  L^{(t)},t\ge 0)$,
see  Theorem  1.4.3  in  \cite{dlg:rtlpsbp}.   The process  $H$  is  the
so-called height-process and $(\ct_H,  d_H)$ is the corresponding L\'evy
tree. Notice that $\bN$-a.e. we have $H_t=0$ for $t\geq \sigma$.

\subsubsection{Local time for the height process and CB}
We now check  that $\ct_H$ represents the genealogy of a
CB with branching mechanism $\psi$. 

The local time of the height process is defined through the next
result, see \cite{dlg:rtlpsbp}, Lemma 1.3.2 and Proposition 1.3.3.
\begin{prop}
\label{prop:LT} 
There exists a jointly measurable process $(L^a_s, a\geq 0, s\geq 0)$
which is continuous and non-decreasing in the variable $s$ such that:
\begin{itemize}
   \item 
For every $t\geq 0$,  $\displaystyle  \lim_{\varepsilon
  \rightarrow 0} \sup_{a\geq 0}  \bE \left[\sup_{s\leq t}
  \left|\varepsilon^{-1} \int_0^s \ind_{\{ a<H_r\leq  a+\varepsilon\}}\;
    dr - L^a_s\right|\right]=0$. 
   \item 
For every $t\geq 0$,  $\displaystyle  \lim_{\varepsilon
  \rightarrow 0} \sup_{a\geq \varepsilon}  \bE  \left[\sup_{s\leq t}
  \left|\varepsilon^{-1} \int_0^s \ind_{\{ a-\varepsilon<H_r\leq  a\}}\;
    dr - L^a_s\right|\right]=0$. 
   \item $\bP$-a.s., for every $t\geq 0$, $L^0_t=-I_t$. 
   \item The occupation time formula holds: for any non-negative
     measurable function $g$ on $\R_+$ and any 
     $s\geq 0$, $\displaystyle \int_0^s g(H_r)\; dr =\int _{(0,{+\infty} )}
     g(a) L^a_s\; da$. 
\end{itemize}
\end{prop}
Let $T_x=\inf\{t\geq 0; I_t\leq -x\}$. We have the following Ray-Knight
theorem which explains why the L\'evy CRT can be viewed as the
genealogical tree of a CB. 
\begin{prop}[\cite{dlg:rtlpsbp}, Theorem 1.4.1]
\label{prop:RK} 
   The process $(L^a_{T_{x}}, a\geq 0)$ is distributed under $\bP$
   as $Y$ under
   $\rP_x$ (i.e. is a CB with branching mechanism $\psi$ starting
   at $x$).
\end{prop}

We then get the following Corollary. 
\begin{cor}
   \label{cor:Y=LH}
   The process  $L(H)=(L^a_{\sigma}, a\geq 0)$ is  distributed under the
   excursion measure $\bN$ as $Y$ under its excursion measure $\N$.
\end{cor}
 
Informally, $L^a_\sigma$ counts  the number of vertices (in fact
leaves) of $\ct_H$ at level $a$ under $\bN$.

\subsubsection{Poissonian representation of the height process above a
  level}
\label{sec:prhpa}

Let $a>0$ be fixed. We consider the excursions of the height process $H$
above $a$ under the excursion measure $\N$. Precisely, let $(u_k,v_k)$,
$k\in \ck$ be the excursions of $H$ above $a$ over the time interval
$[0,\sigma]$. We set $H^k=(H_{{u_k+s}\wedge v_k} -a, s\geq 0)$. 

The next result is a consequence of Proposition 4.2.3 in
\cite{dlg:rtlpsbp}. 
\begin{prop}
   \label{prop:LG-PoisL}
   Conditionally on $(L^r_\sigma, r\leq a)$, the measure $\sum_{k\in \ck
   }  \delta_{H^k}(dH)$  is  a  Poisson  point  measure  with  intensity
   $L^a_\sigma \N[dH]$.
\end{prop}

We give a definition for the number of ancestors, which will be used in
the next 
section. 
\begin{defi}
\label{def:numb-anc}
   The number of ancestors at time $a$ of the
population (coded by $H$) living at time $b$ is the number of excursions
of $H$ above level $a$ which reach level $b>a$: 
\[
R_{a,b}(H)=\sum_{k\in \ck} \ind_{\{\zeta_k\geq  b-a\}},
\]
where $\zeta_k=\max\{H^k_s, s\geq 0\}$. 
\end{defi}

\subsection{Genealogy of $Z$}
Recall  notations from  Section \ref{sec:prhpa}.   In order  to simplify
notations, we  shall write $\N$ for  $\bN$.

We use formulation \reff{eq:def-Z} to construct the genealogy of $Z$. 
Recall notation $\cn_0$ from  Section~\ref{sec:ppm}. 
\begin{itemize}
\item   Conditionally  on   $\cn_0$,   let $\tilde \cn_1(dt,dH)=\sum_{j\in  J_1}
  \delta_{t_j,  H^j} (dt, dH)$  be  a  Poisson point  measure
   with
  intensity $\nu(dt)\; 
\N[dH]$ with $\nu(dt)=\sum_{i\in I} r_i \delta_{t_i}(dt)$.
\item  Let $\tilde \cn_2(dt,dH)=\sum_{j\in J_2} \delta_{t_j, H^j} (dt,
  dH)$ be  a Poisson 
point measure independent of $(\cn_0,\tilde \cn_1)$ and with intensity
$2 \beta \; dt\;  \N[dH]$. 
\end{itemize}
We will write $Y^j$ for $L(H^j)$ for 
$j\in \cj=J_1\bigcup J_2$. Thus notation \reff{eq:def-Z} is still
consistent with the previous Sections, thanks to Corollary
\ref{cor:Y=LH}. And the process $\sum_{j\in \cj} 
\delta_{t_j, H^j}$ allows to code for the genealogy of the families of
$Z$.  

Let $s>0$. Following Definition \ref{def:numb-anc}, we consider $M_s$ the
number of  ancestors at  time $-s$ of  the current population  living at
time $0$, not including the immortal individual:
\[
M_s=\sum_{j\in \cj }\ind_{\{t_j< -s\}} R_{-s-t_j,-t_j}(H^j).
\]

\subsection{Asymptotics for the number of ancestors}

We first give a technical Lemma, which proof is postponed to the end of
this Section.  

\begin{lem}
\label{lem:MZ-dist}
The joint distribution of $M_s$ and $Z_0$ is characterized by the
following equation: for $\eta,\lambda\geq 0$ $s>0$, 
\begin{equation}
\label{eq:MZ-dist}    
\E\left[\expp{-\eta M_s-\lambda Z_0}\right]=\expp{-\int_0^sdr\;
  \tilde{\psi}'(u(\lambda,r))}\E\left[\expp{-Z_{-s}
    [(1-\expp{-\eta})c(s)+\expp{-\eta}u(\lambda,s)]}\right].
\end{equation}
In  particular, $M_s$ has  the same  distribution as  $V^s_{Z_{-s}}$,
where $V^s$ is a Poisson  process with parameter $c(s)$ independent of
$(Z_t, t\in \R)$. 
\end{lem}

\begin{rem}
  Note that one can replace $Z_{-s}$ by $Z_0$ for the right hand side of
  \reff{eq:MZ-dist}   thanks  to   stationarity.  The  effect   of  our
  presentation is to emphasize  the branching property: conditionally on
  $Z_{-s}$, the  number of  families with lifetime  larger than $s$  is a
  Poisson random variable with  parameter the product of population size
  $Z_{-s}$ and the rate  $c(s)=\N(\zeta>s)$ that one family has lifetime
  lager than $s$.
\end{rem}

The next result is the analogue of the result on the number of ancestors
for coalescent process given in \cite{bbl:lcscdi} and \cite{l:scdixcp}.

\begin{theo}
\label{prop:CVNt}
The following convergence holds in probability: 
\[
\lim_{s\rightarrow 0} \frac{M_s}{c(s)}=Z_0.
\]
\end{theo}
\begin{proof}
Let $\rho>0$. We take $\eta=\rho/c(s)$. We deduce from
\reff{eq:MZ-dist} that: 
\[
\lim_{s\rightarrow 0}
 \E\left[\expp{-\rho\frac{ M_s}{c(s)} -\lambda Z_0}\right]
=\E\left[\expp{-Z_0(\rho+\lambda)}\right].
\] 
This implies that $\displaystyle \left(\frac{M_s}{c(s)}, Z_0\right)$
converges in 
distribution to $(Z_0,Z_0)$, which gives the  result.
\end{proof}

\begin{rem}
  Suppose  in  addition   that  $\int_0^\infty  x^2\pi(dx)<\infty$.  Set
  $\tilde{\pi}(dx)=x^2\pi   (dx)$.  Then   the  $\tilde{\pi}$-coalescent
  $N^\mu$ defined  in \cite{bbl:lcscdi} comes down from  infinity by the
  assumption  \reff{eq:A2}  (see  \cite{bbl:lcscdi} and  the  references
  therein). It was  shown in \cite{bbl:lcscdi} that the  speed of coming
  down from infinity satisfies
\begin{equation} 
\label{eq:aux8}
  \lim_{t\downarrow      0+}\frac{N^\mu_t}{c(t)}=1\quad\mbox{     almost
    surely.}
\end{equation}
From the heuristic duality  between coalescence and branching processes,
our  result in  Theorem~\ref{prop:CVNt}  can  be seen  as  a duality  to
(\ref{eq:aux8}).
\end{rem}

\begin{proof}[Proof of lemma \ref{lem:MZ-dist}]
For any $\eta,\lambda\geq 0$, we have: 
\begin{align}
\nonumber
\E\left[\expp{-\eta M_s-\lambda Z_0}\right]
=&\E\left[\expp{-\lambda \sum_{j\in \cj} \ind_{\{-s\leq t_j\leq 0\}}
    Y^j_{-t_j} }\right]\E\left[ 
\exp \left(-\eta M_s-\lambda \sum_{j\in \cj} \ind_{\{t_j<-s\}}
  Y^j_{-t_i}\right)\right]\\ 
\nonumber
=&\exp \left(-\int_0^s dr\; \tilde{\psi}'(u(\lambda,r))\right)\\
\nonumber
&\hspace{1.5cm} \E\left[ \exp\left(-\sum_{j\in \cj} \ind_{\{t_j<-s\}}
    \left(\eta R_{-s-t_j,-t_j}(H^j) +\lambda
      Y^j_{-t_i}\right)\right)\right]\\  
   \label{eq:aux7}
=&\exp \left(-\int_0^s dr\; \tilde{\psi}'(u(\lambda,r))\right)\\
\nonumber
&\hspace{1.5cm} \exp\left(-\int_0^\infty da\; 
\tilde \psi '\left( \N[1- \exp ( - \eta R_{a,a+s }(H)-\lambda Y_{a+s})
]\right)\right),
\end{align}
where  we used  that  Poisson  random measures  over  disjoint sets  are
independent in  the first equality, Lemma \ref{lem:EFfam}  in the second
equality  and a  immediate  generalization of  Lemma \ref{lem:EFfam}  to
genealogies in the third equality.

Using notations from Section \ref{sec:prhpa} on the Poissonian
representation of the height 
process above    level $a$ from Proposition \ref{prop:LG-PoisL}, we get 
\begin{align*}
  \N\left[1- \expp{ - \eta R_{a,a+s }(H)-\lambda Y_{a+s}}\right]
&= \N\left[1- \expp{ - \sum_{k\in \ck} \eta \ind_{\{\zeta_k\geq  s\}} +
    \lambda Y(H^k)_s}\right]\\
&= \N\left[1- \expp{ - Y_a \N\left[1-  \exp\left( -\eta
        \ind_{\{\zeta\geq  s\}} -     \lambda Y_s 
\right)\right]} \right].
\end{align*}
As $\displaystyle  1-  \exp\left( -\eta
        \ind_{\{\zeta\geq  s\}} -     \lambda Y_s 
\right)
=(1- \expp{-\eta})\ind_{\{\zeta\geq  s\}} + \expp{-\eta}(1- \expp{-
  \lambda Y_s })$,
we deduce that 
\[
  \N\left[1- \expp{ - \eta R_{a,a+s }(H)-\lambda Y_{a+s}}\right]
=
\N\left[1- \expp{ - \lambda' Y_a } \right]
=u\left(\lambda', a\right),
\]
with $\lambda'=(1- \expp{-\eta})c(s) +\expp{-\eta}
      u(\lambda,s) $. 
Then we use \reff{lem:Laplace-Z} to write
\begin{align*}
\exp\left(-\int_0^\infty da\; 
\tilde \psi '\left( \N[1- \exp ( - \eta R_{a,a+s }(H)+\lambda Y_{a+s})
]\right)\right) 
&= \exp\left(-\int_0^\infty da\; \tilde \psi
'\left(u(\lambda',a) \right)\right)\\
&=  \E\left[\expp{-\lambda' Z_{-s}}\right]. 
\end{align*}
Plugging this in \reff{eq:aux7}, we get \reff{eq:MZ-dist}.
\end{proof}

\section{The quadratic branching mechanism}
\label{sect:ex}

Let $(\mathbf e_k;k\in \mathbb  N)$ be independent exponential random
 variables with mean $1$. 

\subsection{Preliminaries}
In this Section we give some explicit distributions and more precise
results for the case of quadratic branching mechanism:
\begin{align}
\psi(\lambda)=\beta\lambda^2+ 2\beta\theta\lambda,\label{eq:psi-quadratic}
\end{align}
where $\beta>0$ and $\theta>0$. We have 
\begin{align*}
u(\lambda,t)=\frac{2\theta\lambda}{(2\theta+\lambda)\expp{2\theta\beta
    t} -\lambda},\quad
c(t)=\frac{2\theta}{\expp{2\theta\beta
    t} -1},\quad
\kappa_*=2\theta.
\end{align*}
For every $t\in \mathbb  R$, it follows from Corollary~\ref{cor:cbi}
that the process $\{Z_{s+t};s\geq 0\}$ has the same distribution as the
strong solution of the following stochastic differential equation 
\[
dX_s=\sqrt{2\beta X_s} dW_s+ 2\beta (1-\theta X_s) ds,
\]
with  initial law  $\P(Z_0\in \cdot)$,  where $W$  is a  standard Brownian
motion (see  \cite{ry:cmabm3} Section XI.3  for the existence  of strong
solution).

\subsection{Joint law of the TMRCA and populations sizes} 

We have the following representations. 
\begin{theo}
   \label{theo:law-quad}
 Assume $\psi$ is given by
 \reff{eq:psi-quadratic}. 
\begin{enumerate}[(i)]
\item We have for $\lambda\geq 0$:
\begin{equation}
   \label{eq:Lap-Z-quad}
\E[\expp{-\lambda Z}]=\left(\frac{2\theta}{2\theta +\lambda}\right)^2
 \quad\text{and}\quad
Z\stackrel{(\rm d)}{=} \inv{2\theta} (\mathbf e_1+ \mathbf
e_2).
\end{equation}
   \item We have for $t\geq 0$:
\begin{equation}
   \label{eq:A-quad}
\P(A\leq t)=(1-\expp{-2\theta\beta t})^2
\quad \text{and}\quad A\stackrel{(\rm d)}{=} \inv{2\theta\beta} \max
     (\mathbf e_1, \mathbf e_2).
\end{equation}
\item Conditionally on $\{A=t\}$, we have the following distribution
representation: 
\begin{equation}
\label{eq:ZZAcondA}
\left(Z^A,Z^I,Z^O\right)\stackrel{(\rm d)}{=}\left( \frac{\mathbf
    e_1+\mathbf e_2}{2\theta+c(t)},\frac{\mathbf e_3+\mathbf
    e_4}{2\theta+c(t)},\frac{\mathbf e_5}{2\theta+c(t)}\right) .
\end{equation}
\end{enumerate}
\end{theo}
\begin{proof}
   By Lemma~\ref{eq:Laplace-Z}, we have
\[
\E[\expp{-\lambda Z}]=\left(\frac{2\theta}{2\theta +\lambda}\right)^2. 
\]
This gives (i). 
Using Theorem~\ref{theo:(Z,ZA,ZI,ZO)-dsitrib}, we obtain:
\[
\mathbb  E[\expp{-\lambda Z^A-\gamma Z^I-\eta
  Z^O};A\in dt]=
\frac{2\beta (2\theta)^6 \expp{6\theta\beta
    t}(\expp{2\theta\beta t}-1)}{[(2\theta+\eta)\expp{2\theta\beta
    t}-\eta][(2\theta+\gamma)\expp{2\theta\beta
    t}-\gamma]^2[(2\theta+\lambda)\expp{2\theta\beta t}-\lambda]^2} \; dt. 
\]
We then deduce (ii) and (iii). 
\end{proof}

We  then are  able to  compare more  precisely the  size of  the current
population $Z=Z^I+Z^O$ with the size of the population $Z^A$ just before
the birth  time of the  MRCA. As $(Z_t,  t\in \R)$ is  continuous, notice
that that $Z^A$ is also the size  of the population at the birth time of
the MRCA.   Recall that  $Z^A$ is stochastically  smaller than  $Z$. The
next Corollary indicates that $Z^A$ is however not a.s. smaller than $Z$.

\begin{cor}
   \label{cor:Z-ZA-quad}
 Assume $\psi$ is given by
 \reff{eq:psi-quadratic}. We have: a.s.
\[
   \P(Z^A<Z|A)=\frac{11}{16} \quad\text{and}\quad \E[Z^A|A]=\frac{2}{3} \E[Z|A]
\]
as well as 
\[
   \P(Z^A<Z)=\frac{11}{16} \quad\text{and}\quad \E[Z^A]=\frac{2}{3} \E[Z].
\]
\end{cor}
\begin{proof}
   We have 
\[
\P(Z^A<Z|A)=\P(\mathbf e_1+\mathbf e_2<\mathbf e_3+\mathbf e_4+\mathbf
e_5 )=\frac{11}{16}. 
\]
The other equalities are obvious. 
\end{proof}

There is also an interesting result (which is not valid for general
branching mechanism) which can be interpreted by time reversal. Recall
$\zeta$ is the extinction time of $Y$. 

\begin{prop}
   \label{prop:A-zeta}
Assume $\psi$ is given by
 \reff{eq:psi-quadratic}. Conditionally on $Z$, $A$ is distributed as $\zeta$ under $\rP_Z$: for
all $t\geq 0$
\begin{equation}
   \label{eq:A|Z-quad}
\P(A>t|Z)= \expp{-c(t) Z}=\rP_{Z}(\zeta\leq t).
\end{equation}
\end{prop}

\begin{proof}
  We deduce from \reff{eq:Lap-Z-quad} and \reff{eq:A-quad} that the 
  densities of $Z$ and $A$ are:
\begin{equation}
   \label{eq:densAZ-quad}
f_A(t) =4\theta\beta \expp{-2\theta\beta
  t} (1- \expp{-2\theta\beta t}) \ind_{\{t>0\}}  
\quad\text{and}\quad
f_Z(z) =(2\theta)^2z\expp{-2\theta z}\ind_{\{z>0\}}.
\end{equation}
We also deduce from   \reff{eq:ZZAcondA} the density of $Z$
conditionally on $A=t$:
\[
f_{Z|A=t}(z)=(2\theta+c(t))^3 z^2\expp{-(2\theta+c(t)) z}\ind_{\{z>0\}}.
\]
Using Bayes' rule, we get the density of $A$
conditionally on $Z=z$: for $z,t>0$
\[
f_{A|Z=z}(t)=f_{Z|A=t}(z)\frac{f_A(t)}{f_Z(z)}
=\frac{z(2\theta)^2\beta}{(\expp{2\theta\beta t}-1)^2}\expp{2\theta\beta
  t}\exp\left(-\frac{2\theta z} {\expp{2\theta\beta
      t}-1}\right)=-c'(t)z\expp{-c(t) z}.
\]
We obtain $\P(A\leq t|Z)=\expp{-c(t)Z}$. Then, we conclude as 
\[
\rP_r (\zeta\leq t)=\expp{-r\N[\zeta\geq t]}=\expp{-r c(t)},
\]
where we used the Poissonian representation of $Y$ given by
\reff{eq:Y-rep}. 
\end{proof}

Notice that \reff{eq:A|Z-quad} implies that 
\[
\P(c(A)Z \geq  c(t)Z|Z)=\P(A\leq  t|Z)=\expp{-c(t)Z}.
\]
We obtain  that $c(A)Z$ is independent of $Z$ and 
$ c(A)Z \stackrel{(\rm d)}{=} \mathbf e_1$. We thus deduce the following
Corollary.
\begin{cor}
   \label{cor:ZcAZA-quad}
Assume $\psi$ is given by
 \reff{eq:psi-quadratic}. We have the following representation: 
\[
(Z,
c(A),Z^A)\stackrel{(\rm d)}{=}\left(\frac{\mathbf e_1+\mathbf
    e_2}{2\theta},\; 2\theta\frac{\mathbf
    e_3}{\mathbf e_1+\mathbf e_2},\; 
\inv{2\theta}\frac{\mathbf e_1+\mathbf e_2}{\mathbf e_1+\mathbf
  e_2+\mathbf e_3}(\mathbf e_4+\mathbf e_5)\right). 
\]
\end{cor}

\begin{rem}
  It  is also  easy  to check  that  conditionally on  $\{Z=z\}$, $A$  is
distributed         as         $\displaystyle         \inv{2\beta\theta}
\log\left(1+\frac{2\theta  z}{\mathbf  e_3}\right)$.  In particular,  we
deduce  that  $A$ is  distributed  as $\displaystyle  \inv{2\beta\theta}
\log\left(1+\frac{\mathbf e_1+\mathbf e_2}{\mathbf e_3}\right)$.
\end{rem}

\subsection{TMRCA for $n$ individuals}

Next, we consider the joint distribution of $Z$ and $A^n$ the TMRCA of
the immortal individual and 
$n$ individuals chosen at random among the current population. The next
result is a direct
application of  Theorem \ref{theo:Z,An0}.

\begin{prop}
   \label{prop:An-quad}
Assume $\psi$ is given by
 \reff{eq:psi-quadratic}. We set $s=1-\expp{-2\beta \theta t}$. We have
 for $n\in \N^*$: 
\[
\E\left[Z^n  \expp{-\lambda
    Z}\ind_{\{ A^n\in [0,t]\}} \right]
=\frac{(n+1)! s^n}{(2\theta+\lambda s)^n}
\left(\frac{2\theta}{2\theta+\lambda}\right)^2, 
\]
and the size-biased distribution of $A^n$ is the maximum of
$n$ independent exponential random variables with mean 1:
\[
\E\left[Z^n  \ind_{\{ A^n\in [0,t]\}} \right]
=\E[Z^n] (1-\expp{-2\beta\theta t})^n.
\]
\end{prop}

We  can  compute  explicitly   the  distribution  of  $A^1$.  See  also
\cite{l:ctbp},  section  3,  for  similar  computations  in  a  slightly
different setting.
\begin{prop}
   \label{prop:A1-quad}
Assume $\psi$ is given by
 \reff{eq:psi-quadratic}. We set $s=1-\expp{-2\beta \theta t}$. We have:
\begin{equation}
   \label{eq:A1-quad}
\P(A^1\leq  t)=2\frac{s}{1-s}\left(1+\frac{s}{1-s}\log(s)\right)
\quad \text{and}\quad
\P(c(A^1)Z \geq x|Z)=\frac{2}{x} -\frac{2}{x^2} (1-\expp{-x}).
\end{equation}
In particular $c(A^1)Z$ is independent of $Z$. 
\end{prop}
Notice that $\P(A\leq t)=s^2$ so that we recover from \reff{eq:A1-quad}
the trivial inequality 
$\P(A^1\leq t)\geq \P(A\leq t)$ as $A\geq A^1$. 

\begin{proof}
Applying Theorem~\ref{theo:Z,An0}, we get
\begin{align}
 \E[\expp{-\lambda Z}\ind_{\{A^1\leq t\}}]
=&\int_\lambda^\infty d\eta \; \mathbb  E[Z\expp{-\eta Z}\ind_{\{A^1\leq
  t\}}]\notag\\ 
=&2(\expp{2\theta\beta t}-1)^2
\left(\inv{(\expp{2\theta\beta
      t}-1)}\frac{2\theta}{2\theta+\lambda}-\log\left(1+\inv{(\expp{2\theta\beta
        t}-1)} 
    \frac{2\theta}{2\theta+\lambda}\right)\right)  .\label{eq:ZAn}
\end{align}
In particular,
the distribution of $A^1$ is given by
\[
\P(A^1\leq t)= 2(\expp{2\theta\beta t}-1)^2
\left(\inv{(\expp{2\theta\beta
      t}-1)}-\log\left(1+\inv{(\expp{2\theta\beta t}-1)}\right)\right). 
\]
Applying inverse Laplace transforms to \reff{eq:ZAn} and using
the density of $Z$ given in \reff{eq:densAZ-quad}, we get 
that the conditional law of $A^1$ given $Z$:
\[
\P(A^1\leq t|Z)=\frac{ 2(\expp{2\theta\beta t}-1)^2}{(2\theta)^2 Z}
\left(\frac{2\theta}{\expp{2\theta\beta t}-1}+ 
  \frac{\expp{-2\theta Z/(\expp{2\theta\beta t}-1)} -1}{Z}\right),
\]
which implies that
\[
\P(2\theta Z/(\expp{2\theta\beta A^1}-1)>x)=\frac{2}{x} -\frac{2}{x^2}
(1-\expp{-x}). 
\]   
\end{proof}

\subsection{Fluctuations for the renormalized number of ancestors} 

Finally, we complete Theorem~\ref{prop:CVNt} by giving the fluctuations
for the renormalized number of ancestors. 
\begin{theo}
\label{theo:cvNpi=0}
Assume $\psi$ is given by
 \reff{eq:psi-quadratic}. We have
\[
\sqrt{c(s)\E[Z]}\left(\frac{M_s}{c(s)} -Z\right)
\xrightarrow[s\downarrow   0+]{(\rm   d)} (Z-Z'),
\]
where  $Z'$ is
distributed as $Z$ and independent of $Z$.
\end{theo}
\begin{proof}
We first note that for every $\lambda>0$,
\[
\int_0^s \tilde
  \psi'\left(u(\lambda\sqrt{c(s)},r)\right)\; dr \leq
  s\psi'\left(\lambda \sqrt{c(s)}\right)\xrightarrow[s\downarrow 0+]{} 0, 
\]
and
\begin{align*}
\lim_{s\rightarrow 0}
(1-\expp{\lambda/\sqrt{c(s)}})c(s)+
    \expp{\lambda /\sqrt{c(s)}} u(\lambda\sqrt{c(s)},s)= -\lambda^2/2.
\end{align*}

Under  the   current  assumption   on  the  exponent   $\psi$,  $\mathbb
E[\expp{\lambda  Z}]<\infty$ and $\mathbb  E[\expp{\lambda M_s}]<\infty$
for  $\lambda>0$  small  enough.  Hence,  by  an  analytic  continuation
argument, we  see that (\ref{eq:MZ-dist})  implies that for 
$\lambda>0$  small enough, the following holds for all small $s$:
\begin{equation}
   \label{eq:L(N/c-Z)}
    \E\left[\expp{-\lambda \sqrt{c(s)} \left(Z- \frac{M_s}{c(s)}\right) }\right]
= \expp{-\int_0^s \tilde
  \psi'(u(\lambda\sqrt{c(s)},r))\; dr} \E\left[\expp{- Z
    \left((1-\expp{\lambda/\sqrt{c(s)}})c(s)+
    \expp{\lambda/ \sqrt{c(s)}} u(\lambda\sqrt{c(s)},s)\right)} \right].
\end{equation}
Hence, for all small $\lambda>0$, we
have
\begin{align*}
\lim_{s\rightarrow 0} \E\left[\expp{-\lambda \sqrt{c(s)} (Z- \frac{
      M_s}{c(s)}) }\right]=
\E\left[\expp{\lambda^2Z/2}\right]=\left(\frac{2\theta}{2\theta -
    \lambda^2/2}\right) ^2=\E\left[\expp{-\lambda
    (Z-Z')/\sqrt{\E[Z]}}\right],
\end{align*}
since $\E[Z]=1/\theta$. The result is then a consequence of
\cite{mrs:nmgf}.
\end{proof}

\bibliographystyle{abbrv}
\bibliography{/home/delmas/cermics/Recherche/Bibliographie/delmas}

\newcommand{\sortnoop}[1]{}
\begin{thebibliography}{10}

\bibitem{ad:cmp}
R.~ABRAHAM and J.-F. DELMAS.
\newblock A continuum-tree-valued {M}arkov process.
\newblock {\em arXiv:0904.4175}, 2008.

\bibitem{ad:wdlcrtseppnm}
R.~ABRAHAM and J.-F. DELMAS.
\newblock Williams' decomposition of the {L}\'evy continuous random tree and
  simultaneous extinction probability for populations with neutral mutations.
\newblock {\em Stoch. Process. and Appl.}, 119:1124--1143, 2009.

\bibitem{a:crt2}
D.~ALDOUS.
\newblock The continuum random tree {II}: an overview.
\newblock In {\em Proc. Durham Symp. Stochastic Analysis}, pages 23--70.
  Cambridge univ. press edition, 1990.

\bibitem{a:crt3}
D.~ALDOUS.
\newblock The continuum random tree {III}.
\newblock {\em Ann. Probab.}, 21(1):248--289, 1993.

\bibitem{an:bp}
K.~B. ATHREYA and P.~E. NEY.
\newblock {\em Branching processes}.
\newblock Springer-Verlag, New York, 1972.

\bibitem{bbl:lcscdi}
J.~BERESTYCKI, N.~BERESTYCKI, and V.~LIMIC.
\newblock The $\llambda$-coalescent speed of coming down from infinity.
\newblock {\em ArXiv:0807.4278}, 2009.

\bibitem{bkm:pbss}
J.~BERESTYCKI, A.~E. KYPRIANOU, and A.~MURILLO.
\newblock The prolific backbone for supercritical superdiffusions.
\newblock {\em ArXiv:0912.4736}, 2009.

\bibitem{b:pl}
J.~BERTOIN.
\newblock {\em L\'evy processes}.
\newblock Cambridge University Press, Cambridge, 1996.

\bibitem{blg:sfacpI}
J.~BERTOIN and J.-F. LE~GALL.
\newblock Stochastic flows associated to coalescent processes.
\newblock {\em Probab. Th. Related Fields}, 126(2):261--288, 2003.

\bibitem{blg:sfacpII}
J.~BERTOIN and J.-F. LE~GALL.
\newblock Stochastic flows associated to coalescent processes. {II}.
  {S}tochastic differential equations.
\newblock {\em Ann. Inst. H. Poincar\'e Probab. Statist.}, 41(3):307--333,
  2005.

\bibitem{blg:sfacpIII}
J.~BERTOIN and J.-F. LE~GALL.
\newblock Stochastic flows associated to coalescent processes. {III}. {L}imit
  theorems.
\newblock {\em Illinois J. Math.}, 50(1-4):147--181 (electronic), 2006.

\bibitem{bbcemsw:asbbc}
M.~BIRKNER, J.~BLATH, M.~CAPALDO, A.~ETHERIDGE, M.~M{\"O}HLE, J.~SCHWEINSBERG,
  and A.~WAKOLBINGER.
\newblock Alpha-stable branching and beta-coalescents.
\newblock {\em Electron. J. Probab.}, 10:no. 9, 303--325 (electronic), 2005.

\bibitem{dp:hp}
D.~A. DAWSON and E.~A. PERKINS.
\newblock Historical processes.
\newblock {\em Memoirs of the Amer. Math. Soc.}, 93(454), 1991.

\bibitem{dk:crfvmvd}
P.~DONNELLY and T.~G. KURTZ.
\newblock A countable representation of the {F}leming-{V}iot measure-valued
  diffusion.
\newblock {\em Ann. Probab.}, 24(2):698--742, 1996.

\bibitem{dk:prmvpm}
P.~DONNELLY and T.~G. KURTZ.
\newblock Particle representations for measure-valued population models.
\newblock {\em Ann. Probab.}, 27(1):166--205, 1999.

\bibitem{dlg:rtlpsbp}
T.~DUQUESNE and J.-F. LE~GALL.
\newblock {\em Random trees, {L}\'evy processes and spatial branching
  processes}, volume 281.
\newblock Ast\'erisque, 2002.

\bibitem{dlg:pfalt}
T.~DUQUESNE and J.-F. LE~GALL.
\newblock Probabilistic and fractal aspects of {L}\'evy trees.
\newblock {\em Probab. Th. Rel. Fields}, 131(4):553--603, 2005.

\bibitem{em:ns}
A.~ETHERIDGE and P.~MARCH.
\newblock A note on superprocesses.
\newblock {\em Probab. Th. Related Fields}, 89(2):141--147, 1991.

\bibitem{e:trcs}
S.~N. EVANS.
\newblock Two representations of a conditioned superprocess.
\newblock {\em Proc. Roy. Soc. Edinburgh Sect. A}, 123(5):959--971, 1993.

\bibitem{ep:mvmbpcn}
S.~N. EVANS and E.~PERKINS.
\newblock Measure-valued {M}arkov branching processes conditioned on
  nonextinction.
\newblock {\em Israel J. Math.}, 71(3):329--337, 1990.

\bibitem{f:gtn}
R.~FISHER.
\newblock {\em The genetical theory of natural selection}.
\newblock Clarendon Press, Oxford, 1930.

\bibitem{fl:smvmppgt}
W.~H. FLEMING and M.~VIOT.
\newblock Some measure-valued {M}arkov processes in population genetics theory.
\newblock {\em Indiana Univ. Math. J.}, 28(5):817--843, 1979.

\bibitem{gw:pef}
F.~GALTON and H.~W. WATSON.
\newblock On the probability of the extinction of families.
\newblock {\em J. Roy. Anthropol. Inst.}, 4:138--144, 1874.

\bibitem{g:abctcsbp}
D.~R. GREY.
\newblock Asymptotic behaviour of continuous time, continuous state-space
  branching processes.
\newblock {\em J. Appl. Probability}, 11:669--677, 1974.

\bibitem{js:ccpsvs}
P.~JAGERS and S.~SAGITOV.
\newblock Convergence to the coalescent in populations of substantially varying
  size.
\newblock {\em J. Appl. Probab.}, 41(2):368--378, 2004.

\bibitem{j:sbpcss}
M.~JIRINA.
\newblock Stochastic branching processes with continuous state space.
\newblock {\em Czech. Math. J.}, 83(8):292--312, 1958.

\bibitem{kk:cppsv}
I.~KAJ and S.~M. KRONE.
\newblock The coalescent process in a population with stochastically varying
  size.
\newblock {\em J. Appl. Probab.}, 40(1):33--48, 2003.

\bibitem{kw:bpirll}
K.~KAWAZU and S.~WATANABE.
\newblock Branching processes with immigration and related limit theorems.
\newblock {\em Teor. Verojatnost. i Primenen.}, 16:34--51, 1971.

\bibitem{k:c}
J.~F.~C. KINGMAN.
\newblock The coalescent.
\newblock {\em Stochastic Process. Appl.}, 13(3):235--248, 1982.

\bibitem{l:ctbp}
A.~LAMBERT.
\newblock Coalescence times for the branching process.
\newblock {\em Adv. in Appl. Probab.}, 35(4):1071--1089, 2003.

\bibitem{l:qdcsbpcne}
A.~LAMBERT.
\newblock Quasi-stationary distributions and the continuous-state branching
  process conditioned to be never extinct.
\newblock {\em Electron. J. Probab.}, 12:no. 14, 420--446, 2007.

\bibitem{lg:rrt}
J.-F. LE~GALL.
\newblock Random real trees.
\newblock {\em Ann. Fac. Sci. Toulouse Math. (6)}, 15(1):35--62, 2006.

\bibitem{lglj:bplplfss}
J.-F. LE~GALL and Y.~LE~JAN.
\newblock Branching processes in {L}\'evy processes: Laplace functionals of
  snake and superprocesses.
\newblock {\em Ann. Probab.}, 26:1407--1432, 1998.

\bibitem{lglj:bplpep}
J.-F. LE~GALL and Y.~LE~JAN.
\newblock Branching processes in {L}\'evy processes: The exploration process.
\newblock {\em Ann. Probab.}, 26:213--252, 1998.

\bibitem{l:abctsbp}
Z.-H. LI.
\newblock Asymptotic behaviour of continuous time and state branching
  processes.
\newblock {\em J. Austral. Math. Soc. Ser. A}, 68(1):68--84, 2000.

\bibitem{l:mvbmp}
Z.-H. LI.
\newblock {\em Measure-valued branching {M}arkov processes}.
\newblock Springer, To appear in 2010.

\bibitem{l:scdixcp}
V.~LIMIC.
\newblock On the speed of coming down from infinity for {X}-coalescent
  processes.
\newblock {\em ArXiv:0909.1446}, 2009.

\bibitem{m:cpte}
M.~M{\"O}HLE.
\newblock The coalescent in population models with time-inhomogeneous
  environment.
\newblock {\em Stochastic Process. Appl.}, 97(2):199--227, 2002.

\bibitem{m:rpg}
P.~A.~P. MORAN.
\newblock Random processes in genetics.
\newblock {\em Proc. Cambridge Philos. Soc.}, 54:60--71, 1958.

\bibitem{mrs:nmgf}
A.~MUKHERJEA, M.~RAO, and S.~SUEN.
\newblock A note on moment generating functions.
\newblock {\em Stoch. Prob. Letters}, 76:1185--1189, 2006.

\bibitem{p:cdwpfvp}
E.~A. PERKINS.
\newblock Conditional {D}awson-{W}atanabe processes and {F}leming-{V}iot
  processes.
\newblock In {\em Seminar on {S}tochastic {P}rocesses, 1991 ({L}os {A}ngeles,
  {CA}, 1991)}, volume~29 of {\em Progr. Probab.}, pages 143--156. Birkh\"auser
  Boston, Boston, MA, 1992.

\bibitem{p:ltcsbpi}
M.~A. PINSKY.
\newblock Limit theorems for continuous state branching processes with
  immigration.
\newblock {\em Bull. Amer. Math. Soc.}, 78, 1972.

\bibitem{p:cmc}
J.~PITMAN.
\newblock Coalescents with multiple collisions.
\newblock {\em Ann. Probab.}, 27(4):1870--1902, 1999.

\bibitem{ry:cmabm3}
D.~REVUZ and M.~YOR.
\newblock {\em Continuous martingales and {B}rownian motion}, volume 293.
\newblock Springer Verlag, Berlin Heidelberg New-York, 3 edition, 1999.

\bibitem{rr:pdwcfl}
S.~ROELLY-COPPOLETTA and A.~ROUAULT.
\newblock Processus de {D}awson-{W}atanabe conditionn\'e par le futur lointain.
\newblock {\em C. R. Acad. Sci. Paris S\'er. I Math.}, 309(14):867--872, 1989.

\bibitem{s:gcamal}
S.~SAGITOV.
\newblock The general coalescent with asynchronous mergers of ancestral lines.
\newblock {\em J. Appl. Probab.}, 36(4):1116--1125, 1999.

\bibitem{w:emp}
S.~WRIGHT.
\newblock Evolution in {M}endelian populations.
\newblock {\em Genetics}, 16:97--159, 1931.

\end{thebibliography}
\end{document}